\documentclass{article}

\usepackage[linesnumbered,ruled,vlined]{algorithm2e}
\usepackage{algorithmic}
\usepackage{amssymb}
\usepackage{amsmath}
\usepackage{amsthm}
\usepackage{url}
\usepackage{bbold}
\usepackage{fullpage}

 \newcommand\ForAuthors[1]
 {\par\smallskip                     
  \begin{center}
   \fbox
   {\parbox{0.9\linewidth}
    {\raggedright\sc--- #1}
   }
  \end{center}
  \par\smallskip                     %
 }
 \newcommand\makecomment[1]{}

\def\stab{\operatorname{Sta}}
\def\pd{\operatorname{PD}}
\def\cball#1{\overline{B}(0,#1)}
\def\dset{\mathcal{D}}
\def\urho{\underline{\rho}}
\def\orho{\overline{\rho}}
\def\spef{\mathfrak{f}}
\def\speg{\mathfrak{g}}
\def\speh{\mathsf{H}}
\def\spex{\mathsf{X}}
\def\spev{\mathsf{V}}

\def\funset{\mathbb{\Gamma}}
\def\fset{\Omega([0,1])}
\def\secfunh{\funset_{\rm fun}}
\def\secfunb{\funset_{\rm sc}}

\def\Fu{\mathfrak{F}_u}

\def\Idd{\mathrm{Id}}

\def\rr{\mathbb R}

\def\rd{\rr^d}
\def\nn{\mathbb N}
\def\xin{X^\mathrm{in}}

\def\fin{\operatorname{Fin}\left(\xin\right)}
\def\osp#1{\overline{#1_{\xin}}}
\def\osps#1#2{\overline{#1_{#2}}}
\def\norm#1{\| #1\|}

\def\agmx{\operatorname{Argmax}}

\def\Max{\operatorname*{Max}}
\def\dpcp{\operatorname{DPCP}}

\def\gs{\Delta^s}
\def\klcls{\mathcal{K}\mathcal{L}}
\def\kcls{\mathcal{K}}
\def\kclsi{\mathcal{K}_\infty}
\def\lcls{\mathcal{L}}

\def\bigkse{\mathbf{K}}
\def\bigks{\mathbf{K}^s}

\def\nuopt{\nu_{\rm opt}}

\def\gls#1{{\mathrm R}_{#1}}
\def\gsls#1{\gls{#1}^s}
\def\res{\mathcal{S}}

\newtheorem{assumption}{Assumption}
\newtheorem{prop}{Proposition}
\newtheorem{defi}{Definition}
\newtheorem{theorem}{Theorem}
\newtheorem{corollary}{Corollary}
\newtheorem{problem}{Problem}
\newtheorem{lemma}{Lemma}
\newtheorem{remark}{Remark}
\newtheorem{example}{Example}

\SetKwInOut{Input}{Input}
\SetKwInOut{Output}{Output}
\SetKwInOut{init}{Initialisation}

\title{$\klcls$ and Lyapunov Approaches for Discrete-time Peak Computation Problems}

\author{Assalé Adjé\\ LAboratoire de Modélisation et Pluridisciplinaire et Simulations\\
LAMPS\\ Université de Perpignan Via Domitia\\ France
}

\begin{document}
\maketitle
\begin{abstract}
In this paper, we propose a method to solve discrete-time peak computation problems (DPCPs for short). DPCPs are optimization problems that consist of maximizing a function over the reachable values set of a discrete-time dynamical system. The optimal value of a DPCP can be rewritten as the supremum of the sequence of optimal values. Previous results provide general techniques for computing the supremum of a real sequence from a well-chosen pair of a strictly increasing continuous function on $[0,1]$ and a positive scalar in $(0,1)$. In this paper, we exploit the specific structure of the optimal value of the DPCP to construct such a pair from classical tools from stability theory: $\klcls$ certificate and Lyapunov functions. 
\end{abstract}



\section{Introduction}\label{sec:intro}

In this paper, we propose to solve {\it discrete-time peak computation problems} in finite-dimensional state spaces. A discrete-time peak computation problem can be viewed as a particular maximization problem for which the constraint on the decision variable is to belong to the reachable values set of a given discrete-time dynamical system. The objective function of this optimization problem is a function of the states. This particular maximization problem can model various situations in engineering, economics, physics, biology... For example, in population dynamics, for a predator-prey system, one may be interested in the maximal size of each species, even if the system asymptotically oscillates between several equilibria. This problem boils down to maximizing each coordinate of the state variable separately over the system modeling the population dynamics. 

Besides its importance in analyzing critical concrete situations, the literature on discrete-time peak computation problems in the communities of dynamical systems analysis and optimization is quite thin. Discrete-time peak computation problems appear for linear systems, where the objective function is the Euclidean norm in~\cite{ahiyevich2018upper}. This problem also appears in the analysis of stable linear systems with general quadratic objective functions in~\cite{DBLP:journals/jota/Adje21} and in~\cite{adje13052025}. A more constrained version of the problem appears in~\cite{ahmadi2024robust}. The additional constraint for the state variable is to live in a given set. Discrete-time peak computation problems have a continuous-time counterpart (see e.g.,~\cite{miller2020peaksafety}). Note that in~\cite{miller2020peaksafety}, the authors address the problem of verifying safety properties (see, e.g.,~\cite{belta2017formal} for verification problems on discrete-time systems) for dynamical systems and formulate the safety verification problem as a continuous-time peak computation problem. Safety verification problems aim to prove that the trajectories of a system are contained in a safe set. A safety verification problem can be translated into a peak computation problem when the safe set is a sublevel of sufficiently regular function. In this case, the objective function of the peak computation problem is the function associated with the sublevel set. A safety analysis from the peak computation point of view is also proposed for switched systems in~\cite{adje2017proving} and for program analysis in~\cite{adje2015property}. For the safety analysis of dynamical systems, one can think that a reachability analysis (see e.g.,~\cite{rakovic2006reachability} and references therein), equivalent to a feasibility analysis in our context, is sufficient to prove some properties of the dynamical system. However, maximality in our context is synonymous with criticality. The optimal value in this situation can be viewed as the value that penalizes the system the most with respect to some criteria and defines the extreme conditions under which the system could enter.

In~\cite{miller2020peaksafety}, the authors are interested in computing an upper bound of the peak optimal value for the continuous-time case. The techniques used are based on Sums-Of-Squares (SOS) and Linear Matrix Inequalities (LMI) formulations (the interested reader can consult~\cite{Lasserre_2015} for those subjects). The same goal motivates the authors of~\cite{ahiyevich2018upper}. They limit their study to linear systems and norm objective functions they directly use semidefinite programming (see~\cite{ben2001lectures} for further details on the subject). The purpose of the current paper is to develop a method for computing the exact value of a discrete-time peak computation problem. This purpose is shared by the authors of~\cite{ahmadi2024robust}. However, this latter cannot be fairly compared because of the additional constraint. Moreover, the studied systems are linear or piecewise linear, and the objective functions are linear. 

The goal of the current paper is to provide a theoretical generalization of the works in~\cite{DBLP:journals/jota/Adje21} and ~\cite{adje13052025}. In~\cite{DBLP:journals/jota/Adje21} and in~\cite{adje13052025}, the resolution of discrete-time peak computation problems is restricted to stable affine systems and quadratic objective functions. In this paper, we consider wider classes of stable systems and general objective functions. The techniques developed in this paper combine the method proposed in ~\cite{adje2025maximizationrealsequences} with the Lyapunov approach of~\cite{adje13052025}. In~\cite{adje2025maximizationrealsequences}, the method aims to find the maximal term of a real sequence using homeomorphisms of convergent positive geometric sequences. In~\cite{adje13052025}, homeomorphisms and geometric sequences are constructed from quadratic Lyapunov functions. In this paper, we also propose a construction of homeomorphisms and geometric sequences from $\klcls$ bounds. The use of such classical tools of dynamical systems stability theory (the interested reader can consult~\cite[Chap. 4]{elaydiintroduction} for discrete-time stability notions) to solve an optimization problem appears to be unnatural.  This is a consequence of a simple observation: the existence of a maximizer for a real sequence depends on its asymptotic behavior. More precisely, the existence of a maximizer relies on the finiteness of the limit superior and on the existence of a term strictly greater than the limit superior. Hence, tools for studying the asymptotic behavior of the system help in the existence of a peak and its characterization. Although general stability properties can be difficult to obtain, asymptotic stability properties can be easily presented and understood from $\mathcal{KL}$ bounds (e.g.,~\cite{doi:10.1080/10236190902817844} or~\cite{NESIC20041025}), that is, the norms of the states are globally bounded from above by a $\mathcal{KL}$ function. However, verifying that the $\mathcal{KL}$ function is an upper bound over the norms of the states must be done for all integers, which can be difficult. In contrast, Lyapunov functions are functional certificates for system stability. Some results guarantee the existence of Lyapunov functions: a system is globally asymptotically stable if and only if there exists a Lyapunov function (see, e.g., \cite{khalil2002nonlinear}). Being a Lyapunov function is a timeless property that depends only on the dynamics. In a few advantageous situations, Lyapunov functions can be computed using numerical optimization solvers (see the survey~\cite{giesl2015review}).
 
The paper is organized as follows. Section~\ref{sec:formulation} presents the discrete-time peak computation problems addressed in this paper. In Section~\ref{sec:formulation}, we also present a running example that illustrates the techniques developed in this paper. Section~\ref{sec:preliminaries} recalls the useful results obtained in~\cite{adje2025maximizationrealsequences}. In Section~\ref{sec:preliminaries}, we also provide basic definitions of comparison functions. Then, in Section~\ref{sec:kl}, we develop an approach based on $\klcls$ bounds. In Section~\ref{sec:lyap} we develop the approach using general Lyapunov functions. In these two sections, we explain how to use $\klcls$ and Lyapunov functions to solve a discrete-time peak computation problem. We also apply these techniques to the running example presented in Section~\ref{sec:formulation}. Section~\ref{sec:conclusion} concludes the paper and discusses future research.

\vspace{0,2cm}

\noindent {\bf Notations}: $\rr$ stands for the set of reals, $\rr_+$ the set of nonnegative reals, $\rr^*$ the set of nonzero reals and $\rr_+^*$ the set of strictly positive reals. The vector space of vectors of $d$ reals is denoted by $\rd$. The set of natural integers is denoted by $\nn$ whereas $\nn^*$ denotes the set of nonzero natural integers. The norm $\norm{\cdot}$ stands for the Euclidean norm. Finally, $\cball{\alpha}$ denotes the closed ball of radius $\alpha$ centered at zero i.e., $\cball{\alpha}:=\{x\in\rd : \norm{x}\leq \alpha\}$.




\section{Problem formulation}\label{sec:formulation}

\subsection{Problem formulation}
Let us consider a discrete-time dynamical system $(\xin,T)$ on $\rd$ where
\begin{itemize}
\item $\xin$ is the nonempty subset of $\rd$ of initial conditions;
\item $T$ is a nonzero self-map on $\rd$ that updates the state-variable.  
\end{itemize}
The system is autonomous; it evolves without any control, disturbance, or external input. In this paper, we are interested in the reachable values that maximize a given function. Then, given $\varphi:\rd\to\rd$, we define the {\it discrete-time peak computation problem} $\dpcp(\xin,T,\varphi)$ as follows:
\begin{equation}
\label{pcpprob}
\tag{DPCP}
\begin{array}{lcl}
\displaystyle{\Max_{x,k}} & &\displaystyle{\varphi(T^k(x))}\\
           & \text{ s. t.}& \displaystyle{x\in\xin, k\in\nn}
           \end{array} 
\end{equation}
where $T^k$ is the $k$-fold composition of $T$ if $k$ is not null, and the identity, if $k=0$.
 
Like any optimization problem, to solve a problem of form~\eqref{pcpprob}, we must compute
\begin{itemize}
\item its optimal value i.e., $\sup\{\varphi(T^k(x)): x\in\xin,\ k\in\nn\}$;
\item its optimal solutions i.e., couples $(\overline{x},\overline{k})\in \xin\times \nn$ such that\[\varphi(T^{\overline{k}}(\overline{x}))=   \sup\{\varphi(T^k(x)): x\in\xin,\ k\in\nn\}.\]
\end{itemize}
An optimal solution $(x,k)$ of $\dpcp(\xin,T,\varphi)$ is thus composed of an initial condition $x$ leading to a peak and its associated date $k$ of peak realization.

Without loss of generality, we can assume that $\varphi(0)=0$. Indeed, for the problem $\dpcp(\xin,T,\varphi)$, the optimal value is equal to the sum of the optimal values of $\dpcp(\xin,T,\varphi-\varphi(0))$ and $\varphi(0)$. Moreover, the optimal solutions of $\dpcp(\xin,T,\varphi)$ and $\dpcp(\xin,T,\varphi-\varphi(0))$ are equal.

The optimal value of $\dpcp(\xin,T,\varphi)$ can be described using a sequence of optimal values. Indeed, we can introduce a static classical optimization problem where $k\in\nn$ is fixed:
\begin{equation}
\label{pcpstatic}
\tag{$P_k$}
\begin{array}{lcl}
\displaystyle{\Max_{x}} & &\displaystyle{\varphi(T^k(x))}\\
           & \text{ s. t.}& \displaystyle{x\in\xin}
           \end{array} 
\end{equation}
Then, we introduce the sequence $\nu=(\nu_k)_{k\in\nn}$ defined, for all $k\in\nn$, by:
\begin{equation}
\label{nudef}
\nu_k:=\sup\{\varphi(T^k(x)): x\in\xin\}
\end{equation}
and its supremum:
\begin{equation}
\label{nuopt} 
\nuopt:=\sup\{\nu_k:k\in\nn\}.
\end{equation}
From these notations, $\nuopt$ is equal to the optimal value of $\dpcp(\xin,T,\varphi)$ whereas $\nu_k$ is the optimal value of the static optimization problem~\eqref{pcpstatic}. Solving $\dpcp(\xin,T,\varphi)$ boils down to compute $\nuopt$, an integer $\overline{k}$ such that $\nu_{\overline k}=\nuopt$ and an optimal solution $x\in\xin$ of $(P_{\overline{k}})$. We concentrate our efforts on the computation of the integer $\overline{k}$ as the computation of the optimal solution $x\in\xin$ of $(P_{\overline{k}})$ relies on an optimization solver (when a suitable one is available). The approach proposed in this paper consists in computing a \emph{stopping integer for }  $\dpcp(\xin,T,\varphi)$.
\begin{defi}[Stopping integer]
\label{def:stopping}
An integer $K\in\nn$ is a stopping integer for $\dpcp(\xin,T,\varphi)$ if, for the sequence $\nu$ defined in Equation~\eqref{nudef}:
\[
\nuopt=\max\{\nu_k : k=0,\ldots,K\}.
\]
\end{defi}
If there is no ambiguity, we will simply say that $K$ is a stopping integer for $\nu$.

Obviously, any stopping integer $K$ for $\nu$ is greater than $\overline k$ and $\overline{k}$ is the smallest stopping integer for $\nu$. Following~\cite{adje2025maximizationrealsequences}, if the set $\{j\in\nn : \nu_j>\limsup_{n\to +\infty} \nu_n\}$ is nonempty, a stopping integer for $\nu$ can deduced from the formula:
\begin{equation}
\label{mainformulaaux}
\min_{\substack{k\in \nn\\ \nu_k>h(0)}} \dfrac{\ln(h^{-1}(\nu_k))}{\ln(\beta)}
\end{equation}
where $(h,\beta)$ is such that
\begin{enumerate}
\item $h:\rr\to \rr$ is strictly increasing and continuous on $[0,1]$;
\item $\beta\in (0,1)$;
\item for all $k\in\nn$, $\nu_k\leq h(\beta^k)$;
\item there exists $k\in\nn$ such that $\nu_k>h(0)$.
\end{enumerate}
In this paper, we extend the techniques developed in~\cite{adje13052025} for stable affine systems and quadratic objective functions to nonlinear stable systems and nonquadratic objective functions. In~\cite{adje13052025}, we construct $(h,\beta)$ from quadratic Lyapunov functions. In this paper, we  propose to construct $(h,\beta)$ from any $\klcls$ stability certificate and any Lyapunov function associated with the system $(\xin,T)$. The interested reader can consult~\cite{kelletdcds15} and references therein for complementary readings about those tools from stability theory.     
\subsection{Practical limitations}
From a theoretical perspective, static optimization problems \eqref{pcpstatic} are classical optimization problems; consequently, their study and resolution are relegated to optimization theory. 
For example, when $\varphi \circ T^k$ is concave and $\xin$ is convex, the problem~\eqref{pcpstatic} is said to be convex, and numerous methods exist (see e.g.,~\cite{ben2001lectures,bertsekas2015convex}), or when $\varphi\circ T^k$ is polynomial and $\xin$ a basic semialgebraic set (see e.g.,~\cite{anjos2011handbook,Lasserre_2015}), the problem~\eqref{pcpstatic} falls into the class of polynomial optimization problems.

In practice, the most critical challenge for the computation of $\nu_k$ lies before the optimization procedure. 
Indeed, when $T$ is nonlinear, it is difficult to obtain a closed form for $T^k$. For example, for logistic sequences, closed forms are only known for particular cases (see e.g.,~\cite{gutierrez2010}). Some specific difference equations or systems require intensive efforts to obtain closed forms (see e.g.,~\cite{ELSAYED2012378,articleH,stevic2014representation} and references therein). We could use instead at least an exact formulation of $T^k(\xin)$ when $\xin$ is an infinite set. However, its computation is limited in practice. For example,  in~\cite{magron2015semidefinite}, the authors approximate the polynomial image of a basic semialgebraic set using a sequence of semidefinite programs. The authors obtain the convergence in $L^1$ sense of the sequence of approximations to the polynomial image of the set. 

Hopefully, when $T$ is nonlinear, there exist some situations making the computation of $T^k(\xin)$ "easier" in practice. For example, if $\xin$ is a finite set or $T$ is conjugate to a linear dynamics. Conjugate\footnote{In the literature, the term {\it topologically conjugate} is employed, but in this context, the bijection and its inverse are required to be continuous. In our case, continuity is not mandatory from an optimization perspective.} means that there exists a bijection $B$ and a linear map identified with its matrix $A$ such that for all $x\in\rd$, $T(x)=B^{-1}( A B(x))$. In this case, we simply have for all  $x\in\rd$ for all $k\in\nn$, $T^k(x)=B^{-1}(A^k B(x))$. 

In this paper, the theoretical part does not require finiteness concerning $T$ or $\xin$. We only require the finiteness of $\nu_k$. In our {\it numerical illustrations}, to concentrate on the novelty and avoid entering too deep into optimization tools, we will only consider finite initial condition sets $\xin$. The consideration of more general initial condition sets in practice will be addressed in future works. 

The developments made in this paper require real-valued sequences; hence the only assumption made in this paper is as follows:
\begin{assumption}
\label{mainassum}
For all $k\in\nn,\ \nu_k=\sup\left\{\varphi(T^k(x)) : x\in\xin\right\}$ is finite.
\end{assumption}

A simple situation where Assumption~\ref{mainassum} holds when $\xin$ is a bounded set, $T$ maps from bounded sets to bounded sets, and $\varphi$ is upper semicontinuous.

\subsection{Running example}
\label{subsec:running}
Throughout this paper, we experiment with our main tools on a nonlinear system found in~\cite{li2014computation}. The dynamics is defined as:
\begin{equation}
\label{runndyn}
\speh:\rr^2\ni x\mapsto \dfrac{1}{8}\begin{pmatrix} \norm{x}^2-1 & -1 \\ 1 & \norm{x}^2-1 \end{pmatrix}x
\end{equation}
To write a discrete-time peak computation problem, we also need an initial conditions set and an objective function. Before discussing the choice of the initial conditions set and the objective functions, it should be noted that starting outside a specific ball makes the norm of the state variable blow up. First, by a simple calculus, we obtain, for all $z\in\rr^2$:
\[
\norm{\speh(x)}^2=\dfrac{1}{64}\norm{x}^2\left(1+(\norm{x}^2-1)^2\right).
\]
Then, by introducing the following two functions:
\begin{equation}
\label{auxfun}
\spef:\rr_+\ni s\to \dfrac{1}{64}s(1+(s-1)^2)\ \text{ and } \speg:\rr_+\ni s\to \left\{\begin{array}{cr}\dfrac{\spef(s)}{s} & \text{if } s>0\\ \\
\dfrac{1}{32} & \text{if } s=0
\end{array}\right.,
\end{equation}
we get $\norm{ \speh(x)}^2=\spef(\norm{x}^2)$ for all $x\in\rr^2$. Now as $\norm{\speh(x)}^2/\norm{x}^2=\speg(\norm{x}^2)$ for all nonzero $x\in\rr^2$ and since the assertion ($\speg(s)<1$ and $s>0$) is equivalent to $s\in (0,\sqrt{63}+1)$, we conclude that $\norm{\speh(x)}<\norm{x}$ if and only if $\norm{x}^2\in (0,\sqrt{63}+1)$. Hence, we introduce:
\begin{equation}
\label{radi}
\orho:=\sqrt{63}+1\simeq 8.9373.
\end{equation}
Finally, we have for all $r\in(0,\orho)$, $\speh(\cball{\sqrt{r}})\subseteq \cball{\sqrt{r}}$.    

In future examples, an initial conditions set will be a subset $\spex$ of $\rr^2$ satisfying 
\begin{equation}
\label{initassum}
\sup\{\norm{x} : x\in \spex\}\in (0,\sqrt{\orho}).
\end{equation} 
The exact definition of $\spex$ will depend on the statements that we want to highlight.

An objective function is also required. In the examples, we will analyze the behavior of each coordinate separately; thus, the objective functions are the coordinate functions: 
\begin{equation}
\label{optrun}
\pi_1:\rr^2\ni x=(x_1,x_2)^\intercal \mapsto x_1\text{ and }\pi_2:\rr^2\ni x=(x_1,x_2)^\intercal \mapsto x_2.
\end{equation}
For an initial condition set $\spex$ satisfying the condition~\eqref{initassum}, we will define two discrete-time peak computation problems $\dpcp(\spex,\speh,\pi_1)$ and $\dpcp(\spex,\speh,\pi_2)$. The optimal value of $\dpcp(\spex,\speh,\pi_1)$ is written as
\begin{equation}
\label{eq:dpcpw1}
\omega_{\spex,{\rm opt}}^1:=\sup_{k\in\nn} \omega_{\spex,k}^1 \text{ where } 
\omega_{\spex,k}^1:=\sup_{x\in \spex} \pi_1(\speh^k(x)).
\end{equation}
The optimal value of $\dpcp(\spex,\speh,\pi_2)$ is written similarly:
\begin{equation}
\label{eq:dpcpw2}
\omega_{\spex,{\rm opt}}^2:=\sup_{k\in\nn} \omega_{\spex,k}^2 \text{ where } 
\omega_{\spex,k}^2:=\sup_{x\in \spex} \pi_2(\speh^k(x)).
\end{equation}
The condition~\eqref{initassum} for $\spex$ ensures that for all $k\in\nn$ and all $x\in \spex$, $\pi_1(\speh^k(x))\leq \sup\{\norm{x}: x\in \spex\}$ and $\pi_2(\speh^k(x))\leq \sup\{\norm{x}: x\in \spex\}$. This guarantees that Assumption~\ref{mainassum} holds for the sequences
\begin{equation}
\label{eq:omega}
\omega_\spex^1:=(\omega_{\spex,k}^1)_{k\in\nn} \text{ and }\omega_\spex^2:=(\omega_{\spex,k}^2)_{k\in\nn}.
\end{equation}

\section{Preliminary results}\label{sec:preliminaries}
Before considering the main results of this paper, we recall some useful theoretical results. First, in Subsection~\ref{subsec:realseq}, we bring details about the formula~\eqref{mainformulaaux} obtained in~\cite{adje2025maximizationrealsequences}. Second, in Subsection~\ref{subsec:compare}, we introduce useful material about comparison functions, which constitute the underlying tool behind $\klcls$ and Lyapunov stability. 
\subsection{Results on the supremum of a real sequence}
\label{subsec:realseq}
The developments made in~\cite{adje2025maximizationrealsequences} concern general real sequences, and in this subsection, we consider $u\in\rr^\nn$ for which we aim to compute:
\begin{equation}
\label{realseqpb}
\sup_{n\in\nn} u_n \text{ and } k\in\nn\text{ s.t. } u_k=\sup_{n\in\nn} u_n 
\end{equation}
\subsubsection{Notations and basic results}
We introduce the subset of $\rr^\nn$ consisting of sequences bounded from above: 
\[
\Lambda:=\{u=(u_0,u_1,\ldots)\in\rr^\nn: \sup_{n\in\nn} u_n<+\infty\}.
\]
For an element of $\Lambda$, we are interested in computing the supremum of its terms.
\begin{prop}
\label{suplimsup}
$u\in\Lambda \iff \displaystyle{\limsup_{n\to +\infty} u_n\in\rr\cup\{-\infty\}}$.
\end{prop}

For $u\in\Lambda$, we define $\agmx(u):=\{k\in\nn: u_k=\displaystyle{\sup_{n\in\nn} u_n}\}$ and we introduce, for $u\in\Lambda$, the two sets of ranks:
\begin{equation}
\label{eq:gsls}
\gsls{u}:=\left\{k\in\nn: u_k>\limsup_{n\to +\infty} u_n\right\}\ \text{ and }\ \gs_u:=\left\{k\in\nn: \max_{0\leq j\leq k} u_j>\sup_{j>k} u_j\right\}
\end{equation}
and we define 
\begin{equation}
\label{eq:bigksu}
\bigks_u=\inf\gs_u.
\end{equation}
We recall that the infimum of the empty set is equal to $+\infty$. Hence, $\bigks_u<+\infty$ if and only if $\gs_u\neq\emptyset$. Furthermore, if $\gs_u$ is nonempty, then $\bigks_u\in \gs_u$.
\begin{prop}
\label{argmaxsimple}
Let $u\in\Lambda$. The following assertions hold:
\begin{enumerate}
\item $\gs_u\neq\emptyset\iff \gsls{u}\neq \emptyset$;
\item $\gs_u=\emptyset\iff \sup_{k\in\nn} u_k=\limsup_{n\to +\infty} u_n$;
\item If $\gs_u\neq\emptyset$ then $\bigks_u=\max\agmx(u)$. 
        \end{enumerate}
\end{prop}

\subsubsection{A formula based on pairs of strictly increasing continuous functions and convergent geometric sequences}
\label{bijections}
In this subsection, we present how to construct the formula~\eqref{mainformulaaux}. All details can be found in~\cite{adje2025maximizationrealsequences}. The key tool to compute the supremum of a sequence $u\in\Lambda$ is the existence of a particular sequence $(h(\beta^k))_{k\in\nn}$ that is an upper bound on $u$. The particularity comes from the fact that $h$ is strictly increasing and continuous on $[0,1]$ and $\beta\in (0,1)$. We then introduce the following set of real functions:
\[
\fset:=\{h:\rr\mapsto \rr : \text{ is strictly increasing and continuous on } [0,1] \}.
\]
It should be noted that we can define a function that is strictly increasing and continuous on $[0,1]$ on a set containing $[0,1]$ and then extend it as we want on the entire real line. Therefore, sometimes, we will partially define elements of $\fset$ on a set greater than $[0,1]$ but not everywhere on the real line.
  
We also introduce, for $u\in\rr^{\nn}$ the following sets:
\[
\funset(u):=\{(h,\beta)\in \fset \times (0,1): u_k\leq h(\beta^k),\ \forall\, k\in\nn\} 
\]
and
\[ 
\secfunh(u):=\{h\in \fset : \exists\, \beta\in (0,1) \text{ s.t. } (h,\beta)\in \funset(u)\}\enspace .
\]
Finally, we introduce for $u\in\rr^\nn$ and $h\in\secfunh(u)$:
\[ 
\secfunb(u,h):=\{\beta\in (0,1) : (h,\beta)\in \funset(u)\}\enspace .
\]
As all elements of $\secfunh(u)$  are strictly increasing and continuous on $[0,1]$, they have an inverse on $[0,1]$. For the sake of simplicity, we will simply denote by $h^{-1}$ the inverse of $h\in\fset$. Moreover, the fact that the elements of $\secfunh(u)$ are strictly increasing and continuous allows for simple properties. 
\begin{prop}
\label{hzero}
Let $u\in\Lambda$. Then :
\begin{enumerate}
\item For all $h\in\secfunh(u)$, we have $\limsup_{k\to +\infty} u_k\leq h(0)$;
\item For all $k\in\nn^*$, $u_k<h(1)$ and $u_0\leq h(1)$;
\item If there exists $h\in\secfunh(u)$ such that $u_0=h(1)$ then $u_0=\max_{k\in\nn} u_k$.
\end{enumerate}
\end{prop}
We can also fully characterize the set $\Lambda$ from the nonemptiness of $\funset(u)$.
\begin{prop}
\label{nonemptyfun}
$u\in\Lambda$ if and only if $\funset(u)\neq \emptyset$.
\end{prop}  
In the formula~\eqref{mainformulaaux}, the natural logarithm is defined when $0<h^{-1}(\nu_k))$, which is the same as $h(0)<\nu_k$. Then, we must ensure that the function $h\in\funset(u)$ satisfies $h(0)<\nu_k$ for some $k\in\nn$. This provides the notion of {\it usefulness}.
\begin{defi}[Useful strictly increasing continuous functions]
\label{useful}
Let $u\in\Lambda$ such that $\gsls{u}\neq \emptyset$. A function $h\in \secfunh(u)$ is said to be useful for $u$ if the set 
\begin{equation}
\label{residual}
\res(u,h):=\{k\in\nn : u_k>h(0)\}
\end{equation}
is nonempty. 
By extension, $(h,\beta)\in\funset(u)$ is useful for $u$ if $h$ is useful for $u$.
\end{defi}
From the first statement of Proposition~\ref{hzero}, if $h$ is useful for $u$, there exists $k\in\nn$ such that $u_k>h(0)\geq \limsup_{n\to +\infty} u_n$ i.e., $\gsls{u}\neq \emptyset$. 

The following result aims to prove that there exists an optimal useful affine function for $u$ when $\gsls{u}\neq \emptyset$. 
\begin{theorem}[Existence of a useful optimal affine function] 
\label{funatt}
Let $u\in\Lambda$ such that $\gsls{u}\neq \emptyset$. Then, there exist $a>0$, $b\in (0,1)$ and $c\in\rr$ such that:
\begin{enumerate}
\item $\{k\in\nn : u_k>c\}\neq \emptyset$;
\item $u_k\leq a b^k+c$ for all $k\in\nn$;
\item $u_{\bigks_u}=ab^{\bigks_u}+c$.
\end{enumerate}
\end{theorem}

\begin{corollary}
\label{maincoro}
Let $u\in\Lambda$ such that $\gsls{u}\neq \emptyset$. Then, there exists a pair $(h,\beta)\in\funset(u)$ useful for $u$ such that $u_{\bigks_u}=h(\beta^{\bigks_u})$.
\end{corollary}

\begin{prop}
\label{emptybound}
Let $u\in\Lambda$. Then, $\gsls{u}\neq \emptyset$ if and only if there exists $h\in\secfunh(u)$ useful for $u$.
\end{prop}

Theorem~\ref{funatt} provides a theoretical result and to have in hand an optimal function $h\in \secfunh(u)$ useful for $u$ is very unlikely to happen. However, if we have any function $h\in\secfunh(u)$, we can produce a stopping integer for $u$ which is, in fact, an upper bound of $\bigks_u$. 
To compute it, for $u\in\rr^\nn$, we introduce the function $\Fu:\nn\times \fset\times (0,1)\mapsto \rr\cup\{+\infty\}$ defined for all $k\in\nn$, $h\in\fset$ and $\beta\in (0,1)$ by:
\begin{equation}
\label{mainformula}
\Fu(k,h,\beta):=\left\{
\begin{array}{lr}
\dfrac{\ln(h^{-1}(u_k))}{\ln(\beta)} & \text{ if } (h,\beta)\in\funset(u)\text{ and }k\in\res(u,h)\\
+\infty & \text{ otherwise}
\end{array}
\right.
\end{equation}

\subsubsection{Properties of $\Fu$ and the computation of the supremum of a sequence}
The function defined in~\eqref{mainformula} has useful properties for computing a stopping integer for $u$ as shown in Theorem~\ref{summary3}. To obtain the results of Theorem~\ref{summary3}, we need auxiliary results.
\begin{prop}
\label{mainprop}
Let us take a pair $(h,\beta)\in \funset(u)$ useful for $u$ and let $k\in\res(u,h)$. The following statements are true:
\begin{enumerate}
\item $\Fu(k,h,\beta)$ is well-defined and strictly positive if $u_0<h(1)$ and null if $k=0$ and $u_0=h(1)$;
\item $\lfloor \Fu(k,h,\beta)\rfloor+1=\min\{j\in\nn : h(\beta^j)< u_k\}=\min\{j\in\nn : \beta^j< h^{-1}(u_k)\}$;
\item $\Fu(k,h,\beta)\geq k$;
\item For all $j\in\nn$, $j>\Fu(k,h,\beta)$, $u_j<u_k$;
\item Let $j\in\res(u,h)$. If $u_j\leq u_k$ then $\Fu(k,h,\beta)\leq \Fu(j,h,\beta)$.
\end{enumerate}
\end{prop}

\begin{prop}
\label{simplefactFu}
Let $u\in\Lambda$ such that $\gsls{u}\neq \emptyset$. The following properties hold:
\begin{enumerate}
\item Let $h\in\secfunh(u)$ and $k\in\res(u,h)$. Let $\beta,\beta'\in\secfunb(u,h)$ such that $\beta\leq \beta'$. Then $\Fu(k,h,\beta)\leq \Fu(k,h,\beta')$.
\item Let $g,h\in\secfunh(u)$. Let $k\in\res(u,\min\{g,h\})$. Let $\beta\in\secfunb(u,\min\{g,h\})$. Then $\Fu(k,\min\{g,h\},\beta)\leq \min\{\Fu(k,g,\beta),\Fu(k,h,\beta)\}$.
\end{enumerate}
\end{prop}

The first statement of Proposition~\ref{simplefactFu} can be read as follows: we must choose the smallest possible $\beta\in\secfunb(u,h)$ to reduce the approximation gap. The second statement has a similar meaning in a functional sense: if we have several functions in $\secfunh(u)$, we must consider the minimum functions to reduce the overapproximation gap.

\begin{prop}
\label{inter}
Let $u\in\Lambda$ and $(h,\beta)\in\funset(u)$. Then for all $j\in\nn$, $\bigks_u\leq \lfloor \Fu(j,h,\beta)\rfloor$. Moreover, if $h$ is useful for $u$, we have: \[\min_{k\in\res(u,h)} \Fu(k,h,\beta)=\Fu(\bigks_\nu,h,\beta)\enspace.\]
\end{prop}

We can go further than the results of Proposition~\ref{inter} to obtain a solution to problem~\eqref{realseqpb}.
\begin{theorem}[Solution for~\eqref{realseqpb}]
\label{summary3}
Let $u\in\Lambda$ with $\gsls{u}\neq \emptyset$. Let $(h,\beta)\in\funset(u)$ be useful for $u$ and let $k\in\res(u,h)$. Then :
\[
\max_{n\in\nn} u_n=\max\{u_k : k= 0,\ldots,\lfloor\Fu(\bigks_u,h,\beta)\rfloor\}.
\]   
\end{theorem}

\begin{theorem}[Formula optimality]
\label{formulaopt}
For all $u\in\rr^\nn$, we have the following formula:
\[
\bigks_u = \inf_{(h,\beta)\in\Gamma(u)} \inf_{k\in\res(u,h)} \Fu(k,h,\beta)=  \inf_{(h,\beta)\in\Gamma(u)} \inf_{k\in\res(u,h)}\dfrac{\ln(h^{-1}(u_k))}{\ln(\beta)}
\]
and the $\inf$ can be replaced by $\min$ when $\gsls{u}\neq \emptyset$.
\end{theorem}
We conclude this subsection with a summary in the form of an algorithm (Algorithm~\ref{algomaincst}) to solve the problem~\eqref{realseqpb}.

\begin{algorithm}[h!]
\DontPrintSemicolon

\Input{$u\in\Lambda$ with $\gsls{u}\neq \emptyset$ and $(h,\beta)\in\funset(u)$ useful} 
\Output{$u_{\rm opt}$ and $\bigkse_u=\min\agmx(u)$}
\Begin{
$k=0$; $K=+\infty$, $u_{\rm max}=-\infty$, $k_{\rm max}=0$\;
\While{$k\leq K$}
{
	\If{$k\in\res(u,h)$}{
		\If{$u_{\rm max}< u_k$}{
			$K=\lfloor\Fu(k,h,\beta)\rfloor$\;
			$u_{\rm max}=u_k$\;
			$k_{\rm max}=k$\;
		}
    }
    $k=k+1$\;
}
$u_{\rm opt}=u_{\rm max}$\;
$\bigkse_u=k_{\rm max}$\;
}
\caption{Resolution of the problem described in~\eqref{realseqpb} with a useful $(h,\beta)\in\funset(u)$}
\label{algomaincst}
\end{algorithm}

Algorithm~\ref{algomaincst} terminates as $h$ is supposed to be useful i.e., $\res(u,h)\neq\emptyset$ making $\Fu(k,h,\beta)$ finite for all $k\in\res(u,h)$. We can also improve $\Fu(k,h,\beta)$ when the variable $u_{\rm max}$ is updated. Corollary~\ref{maincoro} ensures that a useful function $h$ exists. 

Returning to the discrete-time peak computation problem~\eqref{pcpprob}, to compute a stopping integer for $\nu$ (defined in~\eqref{nudef}), we need to construct a couple $(h,\beta)$ useful for $\nu$. 

\begin{problem}
\label{computehpb}
Construct a couple $(h,\beta)\in \funset(\nu)$ useful for 
$\nu$ when $\gsls{\nu}\neq \emptyset$.
\end{problem}
The next sections of the paper focus on classical tools from dynamical systems stability theory, such as $\klcls$ certificate and Lyapunov functions. To properly define these tools, we need to introduce the concept of comparison functions.


\subsection{Elements of comparison functions}
\label{subsec:compare}
The global asymptotic stability of discrete-time systems can be studied by means of \textit{comparison functions} (see e.g.,~\cite{GEISELHART201449}). For a comprehensive literature review on comparison functions, interested readers can refer to ~\cite{DBLP:journals/mcss/Kellett14}. In our context, as suggested in Subsection~\ref{subsec:realseq}, the problem~\eqref{pcpprob} can be studied by analyzing the limit superior of the sequence $\nu$ defined in~\eqref{nudef}. Comparison functions will be used to study its limit superior.

  
\begin{defi}[$\mathcal{K}$/$\mathcal{K}_\infty$ functions classes]
A function $\alpha:[0,a)\mapsto\rr_+$ for some $a>0$ belongs to $\kcls$ if $\alpha$ is strictly increasing and continuous on $[0,a)$, and $\alpha(0)=0$. If moreover, $a=+\infty$ and $\lim_{x\to +\infty} \alpha(x)=+\infty$, $\alpha$ is said to belong to $\kclsi$.
\end{defi}

\begin{defi}[$\lcls$ functions class]
A function $\sigma:\rr_+\to \rr_+$ belongs to $\lcls$ if it is continuous, decreasing, and $\lim_{x\to +\infty} \sigma(x)=0$. 
\end{defi}

\begin{defi}[$\klcls$ functions class]
A function $\gamma:\rr_+\times \rr_+\mapsto \rr_+$ belongs to $\klcls$ if 
for all $t\in\rr_+$, $s\mapsto \gamma(s,t)$ belongs to $\kcls$ and 
if for all $s\in\rr_+$, $t\mapsto \gamma(s,t)$ belongs to $\lcls$.
\end{defi}

The class $\kcls$ is included in $\fset$ but $\kcls$ functions are strictly positive except at 0, where they are null. Hence, $h\in\kcls\cap \secfunh(\nu)$ is useful for $\nu$ if at least one term $\nu_k$ is strictly positive. Thus, in this subsection, Assumption~\ref{positivity} must be satisfied 

\begin{assumption}
\label{positivity}
There exist $k\in\nn$ and $x\in\xin$ satisfying $\varphi(T^k(x))>0$.
\end{assumption}
Let us consider the set of positive definite functions on a subset $X\subseteq \rd$ containing 0:
\[ 
\pd(X):=\{f: X \mapsto \rr_+ : f(x)=0\iff x=0\}.
\]
that is the set of nonnegative functions on $X$ null at 0 and nowhere else.

We also need useful results linking positive definite functions and comparison functions. A proof can be found in~\cite{khalil2002nonlinear} (extensions can be found in~\cite{adje2024parametric}). 

\begin{prop}[Lemma 4.3/Appendix C.4~\cite{khalil2002nonlinear}]
\label{khalil}
Let $X\subseteq \rd$ containing 0 in its interior. Let $f$ be a continuous function in $\pd(X)$. Let $r>0$ such that $\cball{r}\subseteq X$. Then, for all $x\in \cball{r}$, we have
\begin{equation}
\label{inegkhalil}
\underline{\alpha}_f(\norm{x})\leq f(x)\leq \overline{\alpha}^f(\norm{x})
\end{equation}
where:
\begin{equation}
\label{defkhalil}
\underline{\alpha}_f:[0,r)\ni s\mapsto \inf_{s\leq \norm{x}\leq r} f(x)\quad \text{ and }\quad \overline{\alpha}^f:[0,r)\ni s\mapsto \sup_{\norm{x}\leq s} f(x)
\end{equation}
and $\underline{\alpha}_f,\overline{\alpha}^f\in\kcls$.

If, moreover, $X=\rd$ and $f(x)$ tends to $+\infty$ as $\norm{x}$ tends to $+\infty$, then $\underline{\alpha}_f,\overline{\alpha}^f\in\kclsi$ and the inequality~\eqref{inegkhalil} holds for all $x\in \rd$.
\end{prop}
Even if $\norm{\cdot}$ represents the Euclidean norm, the result holds for any norm in $\rd$ as they are equivalent.

\section{Solving problem~\ref{computehpb} from \(\klcls\) bounds}
\label{sec:kl}
In this section, we return to the resolution of $\dpcp(\xin,T,\varphi)$ defined in~\eqref{pcpprob}. We recall that the data are the following: 
\begin{itemize}
\item an initial conditions set (nonempty) $\xin\subseteq \rd$;
\item a nonzero map $T:\rd\mapsto \rd$;
\item a function $\varphi:\rd\mapsto \rr$ such that $\varphi(0)=0$. 
\end{itemize}
We recall that we want to compute $\nuopt=\sup\{\nu_k:k\in\nn\}$ where $\nu_k=\sup\{\varphi(T^k(x):x\in\xin\}$. To use the results of Section~\ref{sec:preliminaries} with the sequence $\nu=(\nu_k)_{k\in\nn}$, we suppose that Assumption~\ref{mainassum} holds.

To simplify the notation of the supremum of an objective function $f:\rd\mapsto \rr$ over a nonempty set $X$, we introduce the following new notation:
\[
\osps{f}{X}:=\sup\{f(x) : x\in X\}
\]
and the set of functions that are bounded from above on $\xin$:
\[
\fin:=\{f:\rd\mapsto \rr : \osp{f}<+\infty\}
\]
With these notations, $\nu_k=\osp{(\varphi\circ T^k)}$ and Assumption~\ref{mainassum} is the same as for all $k\in\nn$, $\varphi\circ T^k \in\fin$. 

In Subsection~\ref{subsec:klthe}, we show how to solve Problem~\ref{computehpb} from a $\klcls$ upper bound of the images of the reachable values of the dynamical system $(\xin,T)$. Then, with some additional assumptions, we show that we can also solve Problem~\ref{computehpb} from the classical $\klcls$ upper bound. Classical means that the upper bound is a $\klcls$ upper bound of the norm of the reachable values. The existence of such an upper bound is equivalent to the global asymptotic stability of the system. In Subsection~\ref{subsec:klapply}, we apply the theoretical results obtained in Subsection~\ref{subsec:klthe} to the running example introduced in Subsection~\ref{subsec:running}.

\subsection{Theoretical constructions from $\klcls$ functions}
\label{subsec:klthe}
We propose to combine $\klcls$ upper bounds with the celebrated result obtained by Sontag (recalled in Lemma~\ref{sontag}) to construct a function $h\in\secfunh(\nu)$ such that $e^{-1}\in \secfunb(\nu,h)$.
 
\begin{lemma}[Proposition 7 of \cite{SONTAG199893}]
\label{sontag}
For all $\gamma\in\klcls$, there exists $\theta_1,\theta_2\in\kclsi$ such that for all $s,t\in\rr_+$, $\gamma(s,t)\leq \theta_1(\theta_2(s)e^{-t})$.
\end{lemma}

In Subsection~\ref{subsec:realseq}, we saw the importance of studying the limit superior of a real sequence to compute its maximal term. In general, this study is difficult. Here, we can exploit the particular structure of the sequence $\nu$ and the fact that its terms originate from the reachable values of a dynamical system. The limit superior of $\nu$ can be determined from the $\klcls$ upper bounds. Moreover, this functional approach provides natural functions $h\in\secfunh(\nu)$.
\begin{theorem}[Relaxed $\klcls$ construction]
\label{klclassicmod}
Assume that there exist $\gamma\in \klcls$ and a nonnegative function $\psi\in\fin$ such that $\osp{\psi}>0$ satisfying for all $x\in\xin$ and for all $k\in\nn$
\begin{equation}
\label{classicklmod}
\varphi(T^k(x))\leq \gamma(\psi(x),k)\enspace .
\end{equation}
Then:
\begin{enumerate}
\item $\limsup_{n\to +\infty} \nu_n\leq 0$ and if, moreover,  Assumption~\ref{positivity} holds then $\gsls{\nu}\neq \emptyset$;
\item there exists $h\in\kclsi$ s.t. $(h,e^{-1})\in\funset(\nu)$ and, if, moreover, Assumption~\ref{positivity} holds then $h$ is useful for $\nu$.
\end{enumerate}
\end{theorem}

\begin{proof}
{\itshape 1}. For all $k$, the function $s\mapsto \gamma(s,k)$ is strictly increasing and then for all $x\in\xin$, for all $k\in\nn$, $\gamma(\psi(x),k)\leq \gamma(\osp{\psi},k)$. Taking the supremum over $\xin$, we have $\nu_k\leq \gamma(\osp{\psi},k)$. We conclude by taking the $\limsup$ as $k$ tends to $+\infty$ and recalling that $ \gamma(\osp{\psi},k)$ tends to 0 (as $k$ tends to $+\infty$ as $t\mapsto \gamma(s,t)\in\lcls$ for all $s\in\rr_+$). Finally, under Assumption~\ref{positivity}, there exists $k\in\nn$, such that $\nu_k>0\geq \limsup_{k\to +\infty} \nu_k$.

{\itshape 2}. This is a direct consequence of Sontag's result (Lemma~\ref{sontag} here). Indeed, there exist $\theta_1,\theta_2\in\kclsi$ such that for all $x\in\xin$ and for all $k\in\nn$, $\varphi(T^k(x))\leq \gamma(\psi(x),k)\leq \theta_1(\theta_2(\psi(x))e^{-k})$. Then $\nu_k\leq h(e^{-k}):=\sup_{x\in\xin} \theta_1(\theta_2(\psi(x)e^{-k}))$. Then, as $\theta_1$ and $\theta_2$ are increasing and continuous, $h$ can also be written $\rr_+\ni y\mapsto \theta_1(\theta_2(\osp{\psi}y))$. The function $h$ is thus the composition of $\theta_1$, $ \theta_2$ and $y\mapsto \osp{\psi}y$ that belong to $\kclsi$ (the third as $\osp{\psi}$) and thus belongs to $\kclsi$. Finally, under Assumption~\ref{positivity}, there exists $k\in\nn$ such that $\nu_k>0=h(0)$ and so $h$ is useful in this case. 

\end{proof}

Classically i.e., when $\varphi=\norm{\cdot}$, the inequality $\varphi(T^k(x))\leq \gamma(\psi(x),k)$ (for $\psi=\norm{\cdot}$) is true for all $x\in\rd$ and all $k\in\nn$ if and only if the system $x_{k+1}=T(x_k), x_0\in\rd$ is said to be {\it uniformly global asymptotic stability} (UGAS) (see for example~\cite[Prop. 1]{NESIC20041025} or~\cite[Th. 2]{doi:10.1080/10236190902817844}. This implies that $x_k$ tends to 0 as $k$ tends to $+\infty$. Thus, with a general function $\varphi$, we obtain one-sided information on the asymptotic behavior of the dynamics. However, from Proposition~\ref{khalil} when $\varphi$ is continuous and $\xin$ is bounded, the classical $\klcls$ certificate of stability is sufficient to construct a function $h\in\secfunh(\nu)$ such that $e^{-1}\in \secfunb(\nu,h)$.

\begin{theorem}[Classical $\klcls$ construction]
\label{th:klclassic}
Suppose that $\varphi$ is continuous and $\xin$ is bounded and not reduced to $\{0\}$. Assume that there exists $\gamma\in \klcls$ satisfying for all $x\in\xin$ and for all $k\in\nn$:
\begin{equation}
\label{classickl}
\norm{T^k(x)}\leq \gamma(\norm{x},k)\enspace .
\end{equation}
Then:
\begin{enumerate}
\item $\limsup_{n\to +\infty} \nu_n\leq 0$ and if, moreover,  Assumption~\ref{positivity} holds then $\gsls{\nu}\neq \emptyset$;
\item there exists $h\in\kclsi$ s.t. $(h,e^{-1})\in\funset(\nu)$ and, if, moreover, Assumption~\ref{positivity} holds then $h$ is useful for $\nu$.
\end{enumerate}
\end{theorem}

\begin{proof}
It suffices to obtain an inequality of the form of~\eqref{classicklmod} from the inequality~\eqref{classickl} and then to use Theorem~\ref{klclassicmod}. 
 
If $\varphi$ is continuous, then from Proposition~\ref{khalil},
as $\overline{\varphi}:\rd\ni x\mapsto \max\{\norm{x},\varphi(x)\}$ is continuous, belongs to $\pd(\rd)$ (as $\varphi(0)=0$) and satisfies $\overline{\varphi}(x)$ tends to $+\infty$ as $\norm{x}$ tends to $+\infty$. Hence, we have $\varphi\leq \overline{\varphi}\leq \overline{\alpha}^{\overline{\varphi}}\circ \norm{\cdot}$ with $\overline{\alpha}^{\overline{\varphi}}\in\kclsi$. Then, from inequality~\eqref{classickl}, we obtain $\varphi(T^k(x))\leq \overline{\alpha}^{\overline{\varphi}}(\norm{T^kx})\leq \overline{\alpha}^{\overline{\varphi}}(\gamma(\norm{x},k))$ for all $x\in\xin$. Thus the inequality~\eqref{classicklmod} holds with $\gamma:=\overline{\alpha}^{\overline{\varphi}}\circ \gamma\in \klcls$ and $\psi:=\norm{\cdot}\in\fin$ that verifies $\osp{\psi}>0$ (as $\xin$ bounded and non reduced to $\{0\}$). We conclude from Theorem~\ref{klclassicmod}.

\end{proof}

\begin{remark}
The inequality~\eqref{classickl} forces $T$ to be zero at the origin. This also implies that all sequences $(T^k(x))_k$ starting at $x\in \xin$ converge to 0. The inequality~\eqref{classicklmod} does not imply such properties for $T$. 
\end{remark}

The approach using the functional upper bound detailed in~\eqref{classicklmod} relies on Sontag’s result, which is not completely constructive. Therefore, it would be complicated to construct a function $h\in \secfunh(\nu)$ from $\gamma$ and $\psi$ appearing in~\eqref{classickl} in practice for general systems. Hopefully, we can develop the latter approach to solve, for well-chosen initial condition sets $\spex$, the problems $\dpcp(\spex,\speh,\pi_i)$ presented in Subsection~\ref{subsec:running}.

\subsection{Application to the running example of Subsection~\ref{subsec:running}}
\label{subsec:klapply}
First, we recall that we want to compute, for a subset $\spex$ satisfying the condition~\eqref{initassum}, $\omega_{\spex,{\rm opt}}^1$ and $\omega_{\spex,{\rm opt}}^2$, which are respectively the supremum of the sequences $\omega_\spex^1$ and $\omega_\spex^2$ defined in~\eqref{eq:dpcpw1} and in~\eqref{eq:dpcpw2}. We have to construct for all $i\in\{1,2\}$, a function $h\in\secfunh(\omega_\spex^i)$ such that $e^{-1}\in\secfunb(\omega_\spex^i,h)$. Actually, we will see that we can choose the same function $h$ for 
$\omega_\spex^1$ and $\omega_\spex^2$ and we can construct this function from a $\klcls$ upper bound. From this $h$, we can define $\mathfrak{F}_{\omega_\spex^i}$ for all $i\in\{1,2\}$.
Next, we explicitly define two finite initial condition sets, $\spex$. For each, we define two vectors with numerical values. Then, we apply Algorithm~\ref{algomaincst} to compute the supremum of $\omega_\spex^1$ and $\omega_\spex^2$ for the two finite initial condition sets. 

\subsubsection{Theoretical developments of the running example using a $\klcls$ upper bound function.}
\label{klthrunn}
We observe, since: 
\[
\speg(s)<e^{-1} \text{ and } s>0\iff 0<s<\sqrt{64e^{-1}-1}+1=:\urho,
\]
that $\norm{\speh(x)}^2<e^{-1} \norm{x}^2$ if and only if $0<\norm{x}^2<\urho$. 
Let $r\in (0,\urho)$ and $\{0\}\neq \spex \subseteq \cball{\sqrt{r}}$. Then, for $i\in\{1,2\}$ and for all $z\in \spex$ and $k\in\nn$:
\[
\pi_i(\speh^k(x))\leq \sqrt{\norm{\speh^k(x)}^2}\leq e^{-k/2}\norm{x} \enspace.
\]
Then we have $\pi_i(\speh^k(x))\leq \gamma(\psi(x),k)$ for all $x\in \spex$ and all $k\in\nn$, where:
\[
\gamma(s,t):=s\sqrt{e^{-t}} \text{ and } \psi=\norm{\cdot} \enspace .
\]
The function $\gamma$ is clearly in $\klcls$ and $\psi$ is clearly nonnegative. As $\spex$ is not reduced to $\{0\}$, $\osps{\psi}{\spex}$ is strictly positive.

Assumption~\ref{positivity} holds for both $\pi_1$ and $\pi_2$. Indeed, if $x$ is a nonzero vector such that $\norm{ \speh^k(x)}$ converges to 0, then for all $i\in\{1,2\}$, $\pi_i(\speh^k(x))$ is strictly positive for some integer $k$. A proof is provided in the appendix. The existence of $\klcls$ bound implies that for all $x\in \spex$, $\norm{\speh^k(x)}$ converges to 0.

As mentioned previously, in general, the functions $\theta_1,\theta_2$ of Sontag's result (Lemma~\ref{sontag} here) are not known explicitly. As we have obtained an inequality of the form~\eqref{klclassicmod} where $e^{-1}$ appears explicitly, we have identified $\theta_1,\theta_2$ and we can directly deduce a function $h_\spex^i\in\secfunh(\omega_\spex^i)$ such that $e^{-1}\in\secfunb(\omega_\spex^i,h_\spex^i)$. The same function can be chosen for $i=1$ and $i=2$ and we simply write $h_\spex$ defined as follows:
 \begin{equation}
 \label{eq:hklrun}
 h_\spex:\rr_+\ni s\mapsto \osps{\norm{\cdot}}{\spex}  \sqrt{s}\enspace .
 \end{equation}
As $h_\spex(0)=0$, we have for $i=1,2$, $\res(\omega_\spex^i,h)=\{k\in\nn : \omega_{\spex,k}^i>0 \}$.  It is straightforward to see that
\[
h_\spex^{-1}:\rr_+\ni s\mapsto \left(\dfrac{s}{\osps{\norm{\cdot}}{\spex}}\right)^2 \enspace.
\]
Now since, for all $i\in\{1,2\}$, we dispose of an element of $\funset(\omega_\spex^i)$, we can solve $\dpcp(\spex,\speh,\pi_i)$ using $\mathfrak{F}_{\omega_\spex^i}$ and Algorithm~\ref{algomaincst}. Very simple calculi lead to the formula for all $i\in\{1,2\}$ and all $k$ such that $\omega_{\spex,k}^i>0$:
\[
\mathfrak{F}_{\omega_\spex^i}(k,h,e^{-1})=\dfrac{\ln(h_\spex^{-1}(\omega_{\spex,k}^i))}{\ln(e^{-1})}=2\ln(\osps{\norm{\cdot}}{\spex})-2\ln(\omega_{\spex,k}^i)\enspace .
\]
\subsubsection{Numerical applications.}
\label{klapplirun}
As discussed earlier, we consider two different finite initial condition sets:
\begin{enumerate}
\item $\spex_a:=\{x^1,x^2\}$ where $x^1:=(-1.3,-0.3)^\intercal$ and $x^2:=(-1.1,-0.8)^\intercal$;
\item $\spex_b:=\{y^1,y^2\}$ where $y^1:=(-2.3,0.013)^\intercal$ and $y^2:=(0.7,-2.29)^\intercal$.
\end{enumerate}
To simplify the notations, we write for all $i\in\{1,2\}$, $\omega_{a}^i=(\omega_{a,k}^i)_{k\in\nn}$ and $\omega_{b}^i=(\omega_{b,k}^i)_{k\in\nn}$ instead of $\omega_{\spex_a}^i$ and $\omega_{\spex_b}^i$. Recall that, for all  $i\in\{a,b\}$, $j\in\{1,2\}$ and $k$:
\[
\omega_{i,k}^j:=\sup_{x\in \spex_i} \pi_j(\speh^k(x)).
\]
As $\spex_a$ and $\spex_b$ are now explicit, we can refine the definition of the function $h_\spex$. We write $h_a$ for the function associated with $\spex_a$ and $h_b$ for the one associated with $\spex_2$. Thus, as 
\[
(\osps{\norm{\cdot}}{\spex_a})^2=\max\{\norm{x^1}^2,\norm{x^2}^2\}=1.85\text{ and }
(\osps{\norm{\cdot}}{\spex_b})^2=\max\{\norm{y^1}^2,\norm{y^2}^2\}=5.7341
\]
we get:
\[
h_a^{-1}:s\mapsto \dfrac{s^2}{1.85} \text{ and } h_b^{-1}:s\mapsto \dfrac{s^2}{5.7341}
\]
and we have, for all $j\in\{1,2\}$ and all $k$ such that $\omega_{a,k}^j>0$ :
\[
\mathfrak{F}_{\omega_{a}^j}(k,h_a,e^{-1})=\ln(1.85)-2\ln\left(\max\{\pi_j(\speh^k(x^{1})),\pi_j(\speh^k(x^{2}))\}\right)
\]
and for all $j\in\{1,2\}$ and all $k$ such that $\omega_{b,k}^j>0$:
\[
\mathfrak{F}_{\omega_{b}^j}(k,h_b,e^{-1})=\ln(5.7341)-2\ln\left(\max\{\pi_j(\speh^k(y^{1})),\pi_j(\speh^k(y^{2}))\}\right)
\]
Then, we are able to apply Algorithm~\ref{algomaincst} to compute $\omega_{i,{\rm opt}}^j$ and the maximal integer solution $k_{i,{\rm max}}^j$ for $i\in\{a,b\}$ and $j\in\{1,2\}$.  

\paragraph{The case $\spex_a=\{x^1,x^2\}$ where $x^1=(-1.3,-0.3)^\intercal$ and $x^2=(-1.1,-0.8)^\intercal$}\hphantom{
Let us start with the initial conditions set $\spex_a$.}

To determine the number of images required for the search of the maximum, we need to compute $\mathfrak{F}_{\omega_{a}^j}(k,h_a,e^{-1})$. However, this is only possible if and only if $\omega_{a,k}^j>0$. So we have to wait for the first integer $k_i$ such that $\max\{\pi_j(\speh^{k_i}(x^1)),\pi_j(\speh^{k_j}(x^2))\}$ is strictly positive. The coordinates of $x^1$ and $x^2$ are negative, so we move on and compute: 
\[
\speh(x^1)=\begin{pmatrix}-357/4000\\-767/4000\end{pmatrix} \text{ and } \speh(x^2)=\begin{pmatrix}-27/1600\\-89/400\end{pmatrix}
\]
The vectors still have negative coordinates, and we compute
\[
\speh^2(x^1)=\begin{pmatrix}0.03463\\0.01174\end{pmatrix} \text{ and } \speh^2(x^2)\simeq \begin{pmatrix}0.02982\\0.02432\end{pmatrix}.
\]
 The vectors $\speh^2(x^1)$ and $\speh^2(x^2)$ have strictly positive coordinates, and we can compute our stopping integers:
\[
\begin{array}{ll}
&\displaystyle{\mathfrak{F}_{\omega_{a}^1}(2,h_a,e^{-1})=\ln(1.85)-2\ln\left(\max\{\pi_1(\speh^2(x^1),\pi_1(\speh^2(x^2))\}\right)\approx 7.3415}\\
\text{ and } &\displaystyle{\mathfrak{F}_{\omega_{a}^2}(2,h_a,e^{-1})=\ln(1.85)-2\ln\left(\max\{\pi_2(\speh^2(x^1),\pi_2(\speh^2(x^2))\}\right)\approx 8.0482}
\end{array}
\]
This means that the stopping integer $K$ in Algorithm~\ref{algomaincst} is set to 7 for $\omega_a^1$ and to 8 for $\omega_a^2$ and 
\[
\omega_{a,{\rm opt}}^1=\max\{\omega_{a,k}^1 : k\in\{0,1,\ldots,7\}\} \text{ and }\omega_{a,{\rm opt}}^2=\max\{\omega_{a,k}^2 : k\in\{0,1,\ldots,8\}\}.
\]
 Following Algorithm~\ref{algomaincst}, we compute the $\max\{\pi_j(\speh^k(x^1)),\pi_j(\speh^k(x^2))\}$ until $k$ reaches the stopping integer. As during the process, we have not found greater values than $\omega_{a,2}^i$, we conclude that $\omega_{a,{\rm opt}}^j=\omega_{a,2}^j$ and $k_{a,{\rm max}}^j=2$ for all $j\in\{1,2\}$.

\paragraph{The case $\spex_b=\{y^1,y^2\}$ where $y^1=(-2.3,0.013)^\intercal$ and $y^2=(0.7,-2.29)^\intercal$}\hphantom{
Let us start with the initial conditions set $\spex_b$.}

Contrary to the case $\spex_a$, as $y_1^2=\omega_{b,0}^1=0.7$ and $y_2^1=\omega_{b,0}^2=0.013$, we can compute stopping integers at $k=0$:
\[
\begin{array}{cc}
&\mathfrak{F}_{\omega_{b}^1}(0,h_b,e^{-1})=\ln(5.7341)-2\ln(0.7)\approx 2.4598 \\
\text{ and } &\mathfrak{F}_{\omega_{b}^2}(0,h_b,e^{-1})=\ln(5.7341)-2\ln(0.013)\approx 10.432\enspace .
\end{array}
\]
This leads to the stopping integer $K$ in Algorithm~\ref{algomaincst} being set to 2 for ${\omega_{b}^1}$ and 10 for ${\omega_{b}^2}$. Hence
\[
\omega_{b,{\rm opt}}^1=\max\{\omega_{b,0}^1,\omega_{b,1}^1,\omega_{b,2}^1\} \text{ and }\omega_{b,{\rm opt}}^2=\max\{\omega_{b,k}^2 : k\in\{0,1,\ldots,10\}\}.
\]
Now, we compute the images of $y^1$ and $y^2$ by $\speh$: 
\[
\speh(y^1)\simeq\begin{pmatrix} -1.23505\\-0.28053\end{pmatrix}\text{ and }\speh(y^2)\simeq \begin{pmatrix}0.70048\\-1.26764\end{pmatrix}.
\]
 As $\pi_1(\speh(y^2))>0.7$, we can update the stopping integer for ${\omega_{b}^1}$ and:
\[
\mathfrak{F}_{\omega_{2}^1}(1,h_b,e^{-1})=\ln(5.7341)-2\ln(\pi_1(\speh(y^2)))\approx 2.4584.
\]
Following the fifth statement of Proposition~\ref{mainprop}, this is slightly smaller than $\mathfrak{F}_{\omega_{b}^1}(0,h_b,e^{-1})$ but its integer part is still equal to 2 for ${\omega_{b}^1}$. 
The second coordinate of $\speh(y^1)$ and of $\speh(y^2)$ are negative, then we cannot update the stopping integer for ${\omega_{b}^2}$. Next, we have
\[
\speh^2(y^1)\simeq \begin{pmatrix}-0.05819\\
  -0.17556\end{pmatrix}\text{ and } \speh^2(y^2)\approx \begin{pmatrix}0.25456\\-0.08636\end{pmatrix}.
\] 
As $\pi_1(\speh^2(y^2))$ is strictly less than $\pi_1(\speh(y^2))$ and the stopping integer for $\omega_b^1$ is equal to 2, it follows that $\omega_{b,{\rm opt}}^1=\pi_1(\speh(y^2))$ and $k_{b,{\rm max}}^1=1$. Since $\pi_2(\speh^2(y^1))$ and $\pi_2(\speh(y^2))$ are negative then nothing changes for ${\omega_{b}^2}$. We continue to iterate for $\omega_b^2$ until $k=10$. We have 
\[
\speh^3(y^1)\approx \begin{pmatrix}0.02897\\0.01392\end{pmatrix} \text{ and } \speh^3(y^2)\approx\begin{pmatrix}-0.01873\\ 0.04183\end{pmatrix}
\]
thus $\omega_{b,3}^2=\pi_2(\speh^3(y^2))>0.013$, which was the maximal previous value. We then update the value of the stopping integer $K$ of Algorithm~\ref{algomaincst} for $\omega_b^2$ and obtain
\[
\mathfrak{F}_{\omega_{b}^2}(3,h_b,e^{-1})=\ln(5.7341)-2\ln(\pi_2(\speh^3(y^2)))\approx 8.0945.
\]
The new stopping integer $K$ is equal to 8 for $\omega_b^2$. The next values are smaller than $\pi_2(\speh^3(y^2))$, which implies that $\omega_{b,{\rm opt}}^2=\pi_2(\speh^3(y^2))$ and $k_{b,{\rm max}}^2=3$.
\section{Solving Problem~\ref{computehpb} from Lyapunov functions}
\label{sec:lyap}
Classically, Lyapunov functions are used as certificates of stability (see ~\cite{elaydiintroduction} or~\cite{kelley2001difference}). They bring a constructive approach to prove the stability of the dynamical system through converse Lyapunov theorems, which proved that stability notions are equivalent to the existence of Lyapunov functions~\cite{khalil2002nonlinear}. For some linear, polynomial, or piecewise polynomial systems, Lyapunov functions can be computed using numerical optimization solvers based on linear or semidefinite programming (e.g., see~\cite{giesl2015review} and references therein). 

Classical quadratic Lyapunov functions solve Problem~\ref{computehpb} when $T$ is affine and stable, and $\varphi$ is quadratic (see~\cite{adje13052025}). In this section, we prove that this approach can be generalized to all discrete-time systems for which a Lyapunov function exists.

\subsection{Definition and useful facts}
We begin by recalling the definition of classical Lyapunov functions for discrete-time systems. We briefly present the equivalent definitions of the decrease condition required to be a Lyapunov function. For topological notions, we recall that $\displaystyle{\norm{\cdot}}$ is the Euclidean norm. However, the context of finite-dimensional space makes the results independent of the choice of the norm.

For $F:\rd\mapsto \rd$, we denote by $\stab(F)$ the set of closed subsets of $\rd$ containing 0 but non reduced to 0 and stable by $F$ i.e.,
\[
\stab(F):=\{\dset\subseteq \rd : \dset \text{ closed },\ 0\in\dset\neq \{0\},\ F(\dset)\subseteq \dset\}
\]
Following Lazar~\cite{lazarthesis}, we define local Lyapunov functions on stable sets. We slightly relax the notion of stable sets, called classically positive or forward invariant sets. Indeed, classically, positive invariant sets are our stable sets with an auxiliary condition: 0 belongs to the interior of the stable set. This condition is not considered stability problems in this paper.   
\begin{defi}[Classical Lyapunov functions]
\label{lyapclassic}
A function $V:\rd\to \rr$ is said to be a (local) Lyapunov function for $F:\rd\mapsto\rd$ on $X\in\stab(F)$ if and only if:
\begin{enumerate}
\item There exist $\alpha_1,\alpha_2\in\kcls$ such that, for all $x\in X$: 
\[
\alpha_1(\norm{x})\leq V(x)\leq \alpha_2(\norm{x})
\]
\item There exists $\lambda\in (0,1)$ such that for all $x\in X$: 
\[
V(F(x)) \leq \lambda V(x)
\]
\end{enumerate}
\end{defi}

\begin{example}
\label{ex:lyapratio}
In~\cite{li2014computation}, the authors propose the following function:
\begin{equation}
\label{lyapunovrun}
\spev:\rr^2\ni x\mapsto \max\{\norm{x}^2,e\norm{\speh(x)}^2\}
\end{equation}
as a (local) Lyapunov function for the dynamics $\speh$ of our running example presented in Subsection~\ref{subsec:running}. 

Next, we prove that $\spev$ is a Lyapunov function for $\speh$ on any $\cball{\sqrt{r}}$ for which $r<\orho$. In this example, we prove the first statement of Definition~\ref{lyapclassic}. Let $r$ be in $(0,\orho)$. Thus, from the discussion of Subsection~\ref{subsec:running}, we have $\norm{\speh(x)}^2<\norm{x}^2$ for all nonzero $x$ in $\cball{\sqrt{r}}$. Then, for all $x$ in $\cball{\sqrt{r}}$, $\norm{x}^2\leq \spev(x)\leq \max\{\norm{x}^2,e \norm{x}^2\}=e\norm{x}^2$. We conclude that there exist $\alpha_1,\alpha_2\in\kclsi$ ($\alpha_1:\rr_+\ni s\mapsto s^2$ and $\alpha_1:\rr_+\ni s\mapsto e s^2$) such that $\alpha_1(\norm{x})\leq \spev(x)\leq \alpha_2(\norm{x})$. 

We will prove the second statement by computing the {\it local Lyapunov ratio operator} of $\speh$.

\end{example}
First, due to the equivalence of norms in $\rd$ and since any element of $\kcls$ is increasing if the first statement holds for $\norm{\cdot}$, it holds for all norms on $\rd$. Second, combining the results and discussions found in~\cite[Remark 5]{GEISELHART201449}, \cite[Def. 5, Th.4]{7040251} and~\cite[Lemma 25, Corollary 1]{DBLP:journals/mcss/Kellett14}, the condition that there exists $\lambda\in (0,1)$ such that $V\circ F \leq \lambda V$ can be equivalently replaced by one of those statements:
\begin{enumerate}
\item there exists $\alpha:\rr_+\mapsto \rr_+$ positive definite and continuous such that $V\circ F\leq V-\alpha\circ \norm{\cdot}$;
\item there exists $\rho:\rr_+\mapsto \rr_+$ positive definite, continuous and such that $\Idd-\rho$ is positive definite satisfying $V\circ F\leq \rho\circ V$.
\end{enumerate}
First, we remark that from the first statement of Definition~\ref{lyapclassic}, $V$ is positive definite on any closed ball for which the radius belongs to the domain of $\alpha_1$. If we require $V$ to be continuous and $\dset$ contains 0 in its interior, from Proposition~\ref{khalil}, the first statement of Definition~\ref{lyapclassic} is actually equivalent to $V\in\pd(\dset)$. Second, if there exists a classical Lyapunov function for $F$, $F$ must be null at zero\footnote{Actually, if $F(e)=e$ for a nonzero vector, we could replace $V$ by $V(\cdot-e)$ and the lower and the upper bounds on $V(\cdot-e)$ by respectively $\alpha_1(\norm{\cdot-e})$ and $\alpha_2(\norm{\cdot-e})$}. Finally, from the second statement of Definition~\ref{lyapclassic}, we can define a type of norm operator relative to a Lyapunov function. We propose a general definition of \emph{local positive definite ratio operators} and then specify the results for local Lyapunov functions.
\begin{defi}[Local positive definite ratio operators]
Let $F:\rd \to \rd$ such that $F(0)=0$ and $X\in \stab(F)$. The local positive definite ratio operator of $F$ on $X$ is the following map: 
\label{lyapopnorm}
\begin{equation}
\label{oplyapunov}
\pd(X)\ni W:\to N_{F}^X(W):=\sup_{\substack{x\in X \\ x\neq 0}} \dfrac{W(F(x))}{W(x)}
\end{equation}
\end{defi}
Note that the local positive definite ratio operator $N_{F}^X$ is always nonnegative and can take infinite values.

In Proposition~\ref{ratioprop}, we present some useful properties of the local positive definite ratio operator $N_{F}^X$.
\begin{prop}
\label{ratioprop}
Let $F:\rd \to \rd$ such that $F(0)=0$ and $X\in \stab(F)$. Then, for all $W\in\pd(X)$:
\begin{enumerate}
\item $N_\Idd^X(W)=1$ where $\Idd:\rd\ni x\mapsto x$;
\item for all $x\in X\backslash\{0\}$, $W(F(x))\leq N_F^X(W) W(x)$;
\item for all $k\in\nn^*$, $N_{F^k}^X(W)\leq (N_{F}^X(W))^k$;
\item there exists $\lambda\in (0,1)$ such that for all $x\in X$, $W(F(x))\leq \lambda W(x)$ if and only if $N_F^X(W)\in (0,1)$. 
\end{enumerate}
\end{prop}

\begin{proof}
{\itshape 1}. This is a straightforward application of the definitions.

{\itshape 2}. The result clearly holds if $N_F^X(W)$ is equal to $+\infty$. Then, if $N_F^X(W)$ is finite, directly applying the definition of $N_F^X(W)$ and the positivity of $W$ lead to the result.

{\itshape 3}. The result holds when $N_F^X(W)$ is equal to $+\infty$. Now, suppose that $N_F^X(W)$ is finite. If $N_F^X(W)=0$, then $W(F(x))$ for all $x\in X\backslash\{0\}$. It follows that $F(x)=0$ for all $x\in X$. Finally, $F^k(x)=0$ for all $k\in\nn^*$ and $x\in X$. Then $N_{F^k}^X(W)=0$ and the inequality holds. If $N_{F^k}^X(W)=0$ for some $k\in\nn$ the inequality obviously holds. We then suppose that $N_F^X(W)$ and for all $k\in\nn^*$ $N_F^k(W)$ are strictly positive. We now apply induction. The result obviously holds for $k=1$. Assume that it holds for some $k\in\nn^*$. Let $x\in X\backslash\{0\}$ such that $W(F^{k+1}(x))>0$. It follows that $W(F^k(x))>0$. Then $W(F^{k+1}(x))/W(x)=(W(F^{k+1}(x))/W(F^k(x))) \times (W(F^k(x))/W(x))$. As $X\in\stab(F)$, $F^{k+1}(x)$ and $F^k(x)$ belong to $X$ and are not equal to 0. Finally, $W(F^{k+1}(x))/W(x)\leq N_F^X(W)\times N_{F^k}^X(W)\leq (N_{F}^X(W))^{k+1}$ using the induction hypothesis. We obtain the inequality by taking the supremum over $X\backslash\{0\}$.

{\itshape 4}. It follows readily from the definition. 
\end{proof}

The fourth statement of Proposition~\ref{ratioprop} is not difficult to prove. However, it is worth noting that if there exists $\lambda\in (0,1)$ such that for all $x\in X$, $W(F(x))\leq \lambda W(x)$, then $N_F^X(W)\leq \lambda$. For our applications, this latter inequality is important. Indeed, $N_F^X(W)$ will be used as an element $\beta\in\secfunb(h,\nu)$ for a function $h$ to be determined. Since $\lambda$ also belongs to $\secfunb(h,\nu)$, following Proposition~\ref{simplefactFu}, we will have $\mathfrak{F}_\nu(k,h,N_F^X(W))\leq  \mathfrak{F}_\nu(k,h,\lambda)$. This will be highlighted in Paragraph~\ref{comparekllyap}.

\begin{corollary}
Let $F$ be a self-map on $\rd$. Suppose that $F$ admits a local Lyapunov function $V$ on $X\in\stab(F)$. Then: 
\begin{enumerate}
\item $N_F^X(V)\in (0,1)$;
\item for all $x\in X$, $V(F(x))\leq N_F^X(V) V(x)$;
\item for all $k\in\nn$, $N_{F^k}^X(V)\leq (N_{F}^X(V))^k$.
\end{enumerate}
\end{corollary}

\begin{proof}
The first statement follows from the fourth statement of Proposition~\ref{ratioprop}. The second statement is a reformulation of the second statement of Proposition~\ref{ratioprop}, for which the inequality holds for $x=0$ as $N_F^X(V)$ is finite. The finiteness of $N_F^X(V)$ also extends the third statement of Proposition~\ref{ratioprop}. 

\end{proof}

\begin{example}[The local positive definite ratio operator of $\speh$ at $\spev$ on closed balls]
\label{ratiooprun}
Now, let us consider $r\in (0,\orho)$. To go further into our computations, we must compute the local positive definite ratio operator on $\cball{\sqrt{r}}$ of $\speh$ at the Lyapunov candidate $\spev$ defined in~\eqref{lyapunovrun}. For sake of simplicity, we write $N_\speh^r(\spev)$ instead of $N_\speh^{\cball{\sqrt{r}}}(\spev)$.

First, we suppose that $r\leq \urho$. It follows that $\spev(x)=\norm{x}^2$. Moreover, as $r<\orho$,
we also have $\norm{\speh(x)}^2<\norm{x}^2\leq r$ and so, by assumption, $\norm{\speh(x)}^2\leq \urho$ and $\spev(\speh(x))=\norm{\speh(x)}^2$. Finally, in the case where $r\leq \urho$, we have
\[
N_\speh^r(\spev)=\sup_{\substack{\norm{x}^2\leq r\\ x\neq 0}} \dfrac{\spev(\speh(x))}{\spev(x)}=\sup_{\substack{\norm{x}^2\leq r\\ x\neq 0}} \dfrac{\norm{\speh(x)}^2}{\norm{x}^2}=\sup_{\substack{\norm{x}^2\leq r\\ x\neq 0}}\dfrac{1+(\norm{x}^2-1)^2}{64}=\sup_{\substack{\norm{x}^2\leq r\\ x\neq 0}} \speg(\norm{x}^2)
\]
The function $\speg$ is strictly decreasing on $[0,1]$ and strictly increasing on $[1,+\infty)$ and for all $0<s<2$, $\speg(s)<1/32=\speg(0)=\speg(2)$, hence for all $r\leq \orho$:
 \[
N_\speh^r(\spev)=\left\{\begin{array}{lr}\dfrac{1}{32} & \text{ if } r\in (0,2]\\
\\
\dfrac{1+(r-1)^2}{64} & \text{ if } r\in (2,\urho]\end{array}\right. 
\]
When $r=\urho$, we get the equality $N_\speh^r(V)=\speg(\urho)=e^{-1}$. Futhermore, since $\speg$ is increasing on $(2,\urho]$, $N_\speh^r(\spev)\leq \speg(\urho)=e^{-1}$.

Now, let consider $r\in(\urho,\orho)$ and $x\in\rr^2$ such that $\norm{x}^2\in (\urho,r]$. Then we have $\spev(z)=e\norm{\speh(z)}^2$ and $\spev(\speh(x))=e\norm{\speh^2(x)}$ if $\norm{\speh(x)}^2\geq \urho$ and $\spev(x)=\norm{\speh(x)}^2$ if $\norm{\speh(x)}^2< \urho$. Then:
{ \everymath={\displaystyle}
\[
\begin{array}{ll}
N_\speh^r(\spev)=\sup_{\substack{\norm{x}^2\leq r\\ x\neq 0}} \dfrac{\spev(\speh(x))}{\spev(x)}
&=\max\left\{\sup_{\substack{\norm{x}^2\leq \urho}} \dfrac{\spev(\speh(x))}{\spev(x)},\sup_{\substack{\urho<\norm{x}^2\leq r\\ x\neq 0}} \dfrac{\spev(\speh(x))}{\spev(x)}\right\}\\
&=\max\left\{\dfrac{1}{e},\sup_{\substack{\urho<\norm{x}^2\leq r\\ \norm{\speh(x)}^2\leq \urho}} \dfrac{\spev(\speh(x))}{\spev(x)},\sup_{\substack{\urho<\norm{x}^2\leq r\\ \norm{\speh(x)}^2>\urho}} \dfrac{ \spev(\speh(x))}{\spev(x)}\right\}\\
&=
\max\left\{\dfrac{1}{e},\sup_{\substack{\urho<\norm{x}^2\leq r\\ \norm{\speh(x)}^2\leq \urho}} \dfrac{\norm{\speh(x)}^2}{e\norm{\speh(x)}^2},\sup_{\substack{\urho<\norm{x}^2\leq r\\ \norm{\speh(x)}^2>\urho}} \dfrac{e\norm{\speh^2(x)}^2}{e\norm{\speh(x)}^2}\right\}\\
&=
\max\left\{\dfrac{1}{e},\sup_{\substack{\urho<\norm{x}^2\leq r\\ \norm{\speh(x)}^2>\urho}} \dfrac{1}{64}\left(1+(\norm{\speh(x)}^2-1)^2\right) \right\}\\
&=
\max\left\{\dfrac{1}{e},\sup_{\substack{\urho<\norm{x}^2\leq r\\ \norm{\speh(x)}^2>\urho}} \speg(\spef(\norm{x}^2)) \right\}
\end{array}
\]
}
Let us introduce the set 
\begin{equation}
\label{eq:cstlyapratio}
S_r:=\{s\in (\urho,r] : \spef(s)>\urho\}.
\end{equation} 
We thus are interested in the supremum of $\speg(\spef(\norm{x}^2))$ for all $\norm{x}^2\in S_r$. As $\spef$ is strictly increasing on $\rr_+$, there exists a unique real root denoted by $\beta^*$ for the polynomial function of degree 3, $\spef(\cdot)-\urho$. As $\spef(\urho)=\urho/e<\urho$, we have $\beta^*>\urho$. We do not require the exact value of $\beta^*$. In fact, $S_r$ is nonempty if and only if $\spef(r)>\urho$. Indeed, if there exists $s\in S_r$, then, as $\spef$ strictly increases, $\spef(r)\geq \spef(s)>\urho$ and $r\in S_r$. Conversely, it is obvious that $S_r$ is nonempty if $r\in S_r$. Note that if $S_r$ is empty, the supremum over it is $-\infty$ making $N_\speh^r(\spev)=e^{-1}$ for all $r\in (\urho,\beta^*]$. Suppose that $S_r$ is nonempty. As $\spef(\urho)=\urho/e>1$, $\speg\circ \spef$ is strictly increasing on $S_r$ and for all $s\in S_r$:
\[
\speg(\spef(s))\leq \speg(\spef(r))=\dfrac{1}{64}\left(1+\left(\dfrac{1}{64}r(1+(r-1)^2)-1\right)^2\right)<\speg(\spef(\orho))=1.
\]
Finally, we obtain for all $r<\orho$:
\begin{equation}
\label{eq:runratio}
N_\speh^r(\spev)=\left\{\begin{array}{lr}  \speg(\spef(r)) & \text{ if } r>\urho\text{ and } f(r)>\urho\\
e^{-1} & \text{ if } r>\urho \text{ and }  \spef(r)\leq \urho\\
\dfrac{1+(r-1)^2}{64} & \text{ if } r\in (2,\urho]\\
\dfrac{1}{32} & \text{ if } r\in (0,2] \\
\end{array}\right.
\end{equation}
We conclude that $\spev$ satisfies for all $r<\orho$:
\begin{enumerate}
\item for all $z\in \cball{\sqrt{r}}$, $\norm{x}^2\leq \spev(x)\leq e\norm{x}^2$;
\item there exists $\lambda \in (0,1)$ such that for all $x\in \cball{\sqrt{r}})$, $\spev(\speh(x))\leq \lambda \spev(x)$
\end{enumerate}
and $\spev$ is a Lyapunov function for $\speh$ on all $\cball{\sqrt{r}}$ where $r<\orho$. 

\end{example}

\subsection{Theoretical constructions from Lyapunov functions}
\label{subsec:lyapres}
In this subsection, we explain how to construct a solution for Problem~\ref{computehpb} from a Lyapunov function. We recall that a solution of Problem~\ref{computehpb} is an $(h,\beta)\in \funset(\nu)$ ($\nu$ defined in~\eqref{nudef}) useful when $\gsls{\nu}\neq \emptyset$. First, we establish a direct result derived from Theorem~\ref{klclassicmod} without any continuity condition on $\varphi$. 

We formulate a simple lemma asserting that the supremum of a local Lyapunov function is finite over nonempty bounded sets contained in a suitable stable set. The result would be evident if the local Lyapunov function is (upper semi)continuous, but continuity is not mandatory in Definition~\ref{lyapclassic}. 
\begin{lemma}
\label{finitelyap}
Let $F:\rd\mapsto \rd$. Assume that $F$ admits a local Lyapunov function $V$ on $X\in\stab(F)$. Then, for all nonempty bounded sets $B\subseteq X$, $\osps{V}{B} <+\infty$.
\end{lemma}
\begin{proof}
Let $F$ be a self-map on $\rd$ and $V$ be a local Lyapunov function on $X\in\stab(F)$. Let $B$ be a nonempty bounded set contained in $X$. By definition, there exists $\alpha\in\kcls$, such that $V(x)\leq\alpha(\norm{x})$ for all $x\in X$. Then, as $\alpha$ is strictly increasing, $\osps{V}{B}\leq \sup_{x\in B} \alpha(\norm{x})=\alpha( \sup_{x\in B}\norm{x})=\alpha(\max_{x\in \overline{B}} \norm{x})=\alpha(\norm{\overline{y}})$ for some $\overline{y}\in X$ as $X$ is closed. We conclude that $\alpha$ is defined at $\osps{\norm{\cdot}}{B}$ and thus $\osps{V}{B}<+\infty$.

\end{proof}

\begin{remark}
Lemma~\ref{finitelyap} justifies the closedness of elements on $\stab(U)$. As a matter of fact, the definition of being a Lyapunov function may allow functions $\alpha:[0,c)\ni x\mapsto (c-x)^{-1}-c^{-1}$ as functions $\alpha_1,\alpha_2\in \kcls$. If such a function is defined on the set $\{\norm{x}:x\in X\}$ where $X$ is a closed set, then for all $x\in X$, $\norm{x}<c$. This implies that $X$ is compact, and then $\sup_{x\in X} \norm{x}=\norm{\overline{x}}<c$ for some $\overline{x}\in X$. On the contrary, let $Y=\{x\in\rd : \norm{x}<c\}$. If the function $\alpha:[0,c)\ni x\mapsto (c-x)^{-1}-c^{-1}$ is defined on the set $\{\norm{y}: y\in Y\}$. As $\sup_{y\in Y} \norm{y}=c$, so $\alpha$ is not defined on $\sup_{y\in Y} \norm{y}$.
\end{remark}

\begin{theorem}[Direct Lyapunov construction]
\label{directlyap}
Suppose that $\xin$ is bounded and not reduced to $\{0\}$. Suppose that there exists a Lyapunov function $V$ for $T$ on $\dset\in\stab(T)$ such that $\xin\subseteq \dset$. If 
\begin{equation}
\label{inegvarlyap}
\varphi(x)\leq V(x) \text{ for all } x\in\dset,
\end{equation}
 then:
\begin{enumerate}
\item $\limsup_{n\to +\infty} \nu_n\leq 0$ and if, moreover,  Assumption~\ref{positivity} holds then $\gsls{\nu}\neq \emptyset$;
\item the function:
\[
h:\rr_+\ni s\mapsto s\osp{V}
\]
belongs to $\secfunh(\nu)$ and $N_T^\dset(V)\in \secfunb(\nu,h)$ and, if, moreover, Assumption~\ref{positivity} holds then $h$ is useful for $\nu$.
\end{enumerate}
\end{theorem}  

\begin{proof}
Let $x\in\xin$ and $k\in\nn$. As $\dset\in\stab(T)$ and $\xin\subseteq \dset$, then, $T^k(x)\in \dset$ and, from \eqref{inegvarlyap}, $\varphi(T^k(x))\leq V(T^k(x))$. Moreover, as $V$ is a Lyapunov function, from Lemma~\ref{finitelyap}, $\osp{V}$ is finite and we have $\varphi(T^k(x))\leq N_{T^k}^\dset(V)V(x)\leq (N_T^\dset(V))^k \osp{V}=h((N_T^\dset(V))^k)$. The inequality involving $\nu_k$ follows by taking the supremum over $\xin$. As $V$ is a Lyapunov function on $\dset$, $N_T^\dset(V)\in (0,1)$. This implies that the first statement as $(N_T^\dset(V))^k \osp{V}$ tends to 0. As $\xin\neq \{0\}$, then $\osp{V}>0$ and from Lemma~\ref{finitelyap}, as $\xin$ is bounded, $\osp{V}<+\infty$. We obtain the second statement from this. 

\end{proof}

\begin{remark}
Actually, the function $\gamma(s,t):\rr_+\times \rr_+\ni (s,t)\mapsto s\exp(t \ln( N_{T}^\dset(V)))$ is a $\klcls$ function and as $\xin$ is bounded and not reduced to $\{0\}$, $V$ belongs to $\fin$ and verifies $\osp{V}>0$. Theorem~\ref{klclassicmod} thus applies. As we do not depend on Sontag's lemma, we have an explicit description of the element $(h,\beta)\in\funset(\nu)$.
\end{remark}

If we assume that $\varphi$ is continuous, the inequality~\eqref{inegvarlyap} is no longer required. To recover a link between $\varphi$ and a Lyapunov function, it suffices to use Proposition~\ref{khalil}. 

\begin{theorem}[Continuous Lyapunov construction]
\label{lyapunovversion}
Assume that $\xin$ is bounded and not reduced to $\{0\}$. Suppose that $\varphi$ is continuous and there exists a Lyapunov function $V$ for $T$ on $\dset\in\stab(T)$ such that $\xin\subseteq \dset$. Finally, suppose that there exists $\alpha_V\in\kcls$ satisfying $\alpha_V(\norm{x})\leq V(x)$ for all $x\in\dset$ and $\alpha_V(s)>\osp{V}$ for some $s\in (0,+\infty)$. 
Let $\alpha\in\kclsi$ verifying $\varphi(x)\leq \alpha(\norm{x})$ for all $x\in\rd$. We define:
\[
h:\rr_+\ni s\mapsto \alpha\left(\alpha_V^{-1}(s\osp{V})\right) 
\]
Then $(h,N_T^\dset(V))\in\funset(\nu)$. If Assumption~\ref{positivity} holds, then $h$ is useful for $\nu$.
\end{theorem}

\begin{proof}
Let $x\in\xin$ and $k\in\nn$. As $\xin\subseteq \dset$ and $\dset\in\stab(T)$, $T^k(x)\in\dset$ and thus $\alpha_V(\norm{T^k(x)})\leq V(T^k(x))$. As $V$ is Lyapunov, we then have $\alpha_V(\norm{T^k(x)})\leq V(T^k(x))\leq N_{T^k}^\dset(V) V(x)\leq N_T^{\dset}(V)^k V(x)$. As $x\in\xin$, it follows that $\alpha_V(\norm{T^k(x)})\leq N_T^{\dset}(V)^k\osp{V}$. By assumption, $\osp{V}<\alpha_V(s)$ for some $s\in (0,+\infty)$. We also have $N_T^{\dset}(V)^k\in (0,1)$. Then,
$\alpha_V^{-1}(N_T^{\dset}(V)^k\osp{V}))$ exists. Since $\varphi\leq \alpha\circ \norm{\cdot}$ on $\rd$ with $\alpha$ increasing, the inequality $\varphi(T^k(x))\leq \alpha(\norm{x})\leq \alpha(\alpha_V^{-1}((N_T^{\dset}(V)^k\osp{V})$ holds. Taking the supremum over $\xin$ leads to $(h,N_T^\dset(V))\in\funset(\nu)$. As $\alpha,\alpha_V\in\kcls$, $h(0)=0$ and thus $h$ is useful if Assumption~\ref{positivity} holds.

\end{proof}

\begin{remark}
If we impose the strongest condition: $\alpha_V(s) > \osps{V}{\dset}$ for some $s\in (0,+\infty)$ (instead of $\alpha_V(s) > \osp{V}$), we can directly use Theorem~\ref{directlyap}. Indeed, if this stronger condition holds, we can deduce $\tilde{\varphi}:=\alpha_V(\alpha^{-1}(\varphi(x))\leq V(x)$ for all $x\in\dset$.
\end{remark}

\subsection{Numerical applications to our running example of Subsection~\ref{subsec:running}}

We propose to illustrate the theoretical results obtained in Subsection~\ref{subsec:lyapres} on the running example presented in Subsection ~\ref{subsec:running}. First, we explicit the function $h\in\funset(\omega_\spex^i)$. Then, we compare the results obtained from the Lyapunov approach with those obtained with $\klcls$ upper bounds. Finally, we complete the numerical experiments with new data not treated with the $\klcls$ approach.  

\subsubsection{Theoretical developments of the running example using $\spev$}
We can return to the discrete-time peak computation problems presented in Subsection~\ref{subsec:running}. Recall that the objective functions are the coordinate functions $\pi_1$ and $\pi_2$. It is evident that $\varphi(x)\leq \norm{x}=\alpha_\varphi(\norm{x})$ for all $x\in\rr^2$ where $\alpha_\varphi:\rr_+\ni :s\mapsto s$. Recall that $\spev$ defined in~\eqref{lyapunovrun} is proved in Example~\ref{ratiooprun} to be a Lyapunov function for $\speh$ on all $\cball{\sqrt{r}}$ where $r<\orho$. It follows that $\alpha_\spev(\norm{x})=\norm{x}^2\leq \spev(x)$ for all $x\in\rr^2$ where $\alpha_\spev:\rr_+\ni s\mapsto s^2$. 

Now let $r<\orho$. Suppose that $\spex$ is such that $\spex\subseteq \cball{\sqrt{r}}$ and there exists $x\in \spex$ with $\norm{x}^2=r$. As the computation depends on $r$, we write $\osps{\spev}{\spex}_r=\sup_{x\in\spex} V(x)$. Then, by definition of $\spev$, $\osps{\spev}{\spex}_r=\sup_{x\in\spex} \max\{\norm{x}^2,e\norm{\speh(x)}^2\}$. By assumption on $\spex$, $\osps{\spev}{\spex}_r=\max\{r,\spef(r)\}$. Now, following Theorem~\ref{lyapunovversion}, 
\begin{equation}
\label{eq:hlyaprun}
h_{\spev,r}:\rr^+\ni s\mapsto \sqrt{s\osps{\spev}{\spex}_r}
\end{equation}
satisfies $(h_{\spev,r},N_\speh^r(\spev))\in\funset(\omega_\spex^i)$ for all $i\in\{1,2\}$. This means that for all $k\in\nn$ 
\[
\omega_{\spex,k}^1=\sup_{x\in\spex} \pi_1(\speh^k(x))\leq \sqrt{N_\speh^r(\spev)\osps{\spev}{\spex}_r}\text{ and } \omega_{\spex,k}^2=\sup_{x\in\spex} \pi_2(\speh^k(x))\leq \sqrt{N_\speh^r(\spev)\osps{\spev}{\spex}_r}\enspace.
\] 
Finally, it is easy to see that \[h_{\spev,r}^{-1} :\rr_+\ni s\to s^2/\osps{\spev}{\spex}_r.\]
Since $h_{\spev,r}(0)=0$, we only apply $h_{\spev,r}^{-1}$ to positive terms $\omega_{\spex,k}^1$ and $\omega_{\spex,k}^2$.
\subsubsection{Comparison with $\klcls$ upper bounds results}
\label{comparekllyap}
In Subsection~\ref{klapplirun}, we considered the case where $\spex\subseteq \cball{\sqrt{r}}$ with $r\leq \urho$. We recall that $\norm{x}^2<\urho$ is equivalent to $\norm{\speh(x)}^2<e^{-1}\norm{x}^2$, which is the same as $e\norm{\speh(x)}^2<\norm{x}^2$. Then, let $r<\urho$ and $\spex\subseteq \cball{\sqrt{r}}$. We have for all $x\in\spex$, $\spev(x)=\norm{x}^2$ and thus $\osps{\norm{\cdot}}{\spex}=\osps{\spev}{\spex}_r$. We conclude that $h_{\spev,r}$ defined in~\eqref{eq:hlyaprun} and $h_r$ defined in~\eqref{eq:hklrun} are equal. Moreover, in Subsection~\ref{klapplirun}, we chose $(h_r,e^{-1})\in\funset(\omega_\spex^i)$ whereas we choose here $(h_{\spev,r},N_\speh^r(\spev))\in\funset(\omega_\spex^i)$. We proved that in Example~\ref{lyapunovrun} in the case where $r<\urho$, that $N_\speh^r(\spev)<e^{-1}$. From the first statement of Proposition~\ref{simplefactFu}, we have for all $k\in\nn$ such that $\omega_{\spex,k}^i>0$,
$\mathfrak{F}_\omega(k,h_{\spev,r},N_\speh^r(\spev))\leq \mathfrak{F}_\omega(k,h_r,e^{-1})$.

For example, we go back the set $\spex_a=\{x^1,x^2\}$ where $x^1=(-1.3,-0.3)^\intercal$ and $x^2=(-1.1,-0.8)^\intercal$. We have $\spex_a\subseteq \cball{\sqrt{1.85}}$ with $\norm{x^1}^2=1.85<\urho$. Then $h_r=h_{\spev,r}:\rr_+\ni s:\mapsto s^2/1.85$. However, in Example~\ref{ratiooprun}, the value of $N_\speh^{1.85}(\spev)=1/32<e^{-1}$ and
\[
\begin{array}{lcl}
\mathfrak{F}_{\omega_a^1}(2,h_{\spev,r},N_\speh^{1.85}(\spev))&=&\dfrac{\ln(1.85)-2\ln\left(\max\{\pi_1(\speh^2(x^1),\pi_1(\speh^2(x^2))\}\right)}{\ln(32)}\approx 2.1183\\ 
&<& \mathfrak{F}_{\omega_a^1}(2,h_r,e^{-1})\approx 7.3415\enspace.
\end{array}
\]
This directly provides the stopping integer equal to 2 and saves computations for $\omega_{a}^1$.
We also have
\[
\begin{array}{lcl}
\mathfrak{F}_{\omega_a^2}(2,h_{\spev,r},N_\speh^{1.85}(\spev))&=&\dfrac{\ln(1.85)-2\ln\left(\max\{\pi_2(\speh^2(x^1),\pi_2(\speh^2(x^2))\}\right)}{\ln(32)}\approx 2.3222\\ 
&<& \mathfrak{F}_{\omega_a^2}(2,h_r,e^{-1})\approx 8.0482\enspace .
\end{array}
\]
Again, we save computations with respect to $\mathfrak{F}_{\omega_a^2}(2,h_r,e^{-1})$.

The case $\spex_b=\{y^1,y^2\}$ where $y^1=(-2.3,-0.013)^\intercal$ and $y^2=(0.7,-2.29)^\intercal$ is less significative as 5.7341 is close to $\urho\approx 5.7481$. We have $\spex_b\subseteq \cball{\sqrt{5.7341}}$ with $\norm{y^2}^2=5.7341<\urho$. Then $h_r=h_{\spev,r}:\rr_+\ni s:\mapsto s^2/5.7341$. However, in Example~\ref{ratiooprun}, the value of $N_\speh^{5.7341}(\spev)=\dfrac{1+(5.7341-1)^2}{64}\approx 0.36581<e^{-1}\approx 0.36788$ and:
\[
\mathfrak{F}_{\omega_b^1}(1,h_{\spev,r},N_\speh^{5.7341}(\spev))=\dfrac{2\ln(\pi_1(\speh(y^2))-\ln(5.7341)}{\ln(N_\speh^{5.7341}(\spev))}\approx 2.44459 < \mathfrak{F}_{\omega_b^1}(1,h_r,e^{-1})\approx 2.4584;
\]
We also have
\[
\mathfrak{F}_{\omega_b^2}(3,h_{\spev,r},N_\speh^{5.7341}(\spev))=\dfrac{2\ln(\pi_2(\speh^3(y^1)))-\ln(5.7341)}{\ln(N_\speh^{5.7341}(\spev))}\approx 8.04905 < \mathfrak{F}_{\omega_b^2}(3,h_r,e^{-1})\approx 8.0945.
\]
\subsubsection{Numerical applications}
In Subsection~\ref{klapplirun}, we are limited to $\spex\subseteq \cball{\sqrt{r}}$ where $r<\urho$. Considering such a radius $r$, we obtain an upper bound of the form $\varphi(\speh^k(z))\leq e^{-k/2}\norm{z}$. Now, we can consider a greater radius $r$. We cannot exceed $\orho$ since the norm explodes starting with vectors of norm greater than $\sqrt{\orho}$. As in Subsection~\ref{klapplirun}, we consider two initial condition sets: the first is a singleton for which the norm of the vector is greater than $\sqrt{\urho}$. This will explain the difference from the cases developed in Subsection~\ref{klapplirun}. In the second case, we start with two vectors with strictly negative coordinates, the norms of which are greater than $\sqrt{\urho}$. Let us define
\begin{enumerate}
\item $\spex_c=\{z^0\}$ where $z^0=(2.2,-2)^\intercal$;
\item $\spex_d=\{z^1,z^2\}$ where $z^1=(-2.3,-1.9)^\intercal$ and $z^2=(-2.5,-1.5)^\intercal$.
\end{enumerate}

Again to simplify the notations, we write for all $i\in\{1,2\}$, $\omega_{c}^i=(\omega_{c,k}^i)_{k\in\nn}$ and $\omega_{d}^i=(\omega_{d,k}^i)_{k\in\nn}$ instead of $\omega_{\spex_c}^i$ and $\omega_{\spex_d}^i$. Recall that, for all  $i\in\{c,d\}$, $j\in\{1,2\}$ and $k\in\nn$:
\[
\omega_{i,k}^j:=\sup_{x\in \spex_i} \pi_j(\speh^k(x)).
\]
We can refine the definition of $h_{\spev,r}$ and $N_\speh^r(\spev)$. We write $h_c$ for the function associated with $\spex_c$ and $h_d$ for the one associated with $\spex_d$. First, the smallest $r$ satisfying $\spex_c\subseteq \cball{\sqrt{r}}$ is equal to $r^c:=\norm{z^0}^2=8.84>\urho$. The smallest $r$ satisfying $\spex_d\subseteq \cball{\sqrt{r}}$ is equal to $r^d=\max\{\norm{z^1}^2,\norm{z^2}\}=8.9>\urho$. In both cases, the radius exceeds $\urho$, then $\osps{\spev}{\spex_c}=e\norm{\speh(z^0)}^2\approx 23.4535$ and $\osps{\spev}{\spex_d}=e\max\{\norm{\speh(z^1)}^2,\norm{\speh(z^2)}^2\}=e\norm{\speh(z^1)}^2\approx 23.9697$. To know the exact value of $N_\speh^{r^c}(\spev)$ and $N_\speh^{r^d}(\spev)$, we have to check whether $\spef(r^c)>\urho$ and  $f(r^d)>\urho$. Since $\spef(r^c)\approx 8.6281>\urho$ and $\spef(r^d)\approx 8.81795>\urho$, we conclude that $N_\speh^{r^c}(\spev)=\speg(\spef(r^c))=\speg(\spef(8.84))\approx 0.9248$ and $N_\speh^{r^d}(\spev)=\speg(\spef(r^d))=\speg(\spef(8.9))\approx 0.9706$.
Thus, we get:
\[
h_c^{-1}:s\mapsto \dfrac{s^2}{e\norm{\speh(z^0)}^2} \text{ and } h_d^{-1}:s\mapsto \dfrac{s^2}{e\norm{\speh(z^1)}^2}
\]
and we have, for all $j\in\{1,2\}$ and all $k$ such that $\omega_{c,k}^j>0$ :
\[
\mathfrak{F}_{\omega_{c}^j}(k,h_c,N_\speh^{8.84}(\spev))=\dfrac{2\ln(\pi_j(\speh^k(z^{0}))-2\ln(\norm{\speh(z^0)})-1}{\ln(\speg(\spef(8.84)))}
\]
and for all $j\in\{1,2\}$ and all $k$ such that $\omega_{d,k}^j>0$:
\[
\mathfrak{F}_{\omega_{d}^j}(k,h_d,N_\speh^{8.9}(\spev))=\dfrac{2\ln(\max\{\pi_j(\speh^k(z^{1})),\pi_j(\speh^k(z^{2}))\})-2\ln(\norm{\speh(z^1)})-1}{\ln(\speg(\spef(8.9)))}
\]
Then, we are able to apply Algorithm~\ref{algomaincst} to compute $\omega_{i,{\rm opt}}^j$ and the maximal integer solution $k_{i,{\rm max}}^j$ for $i\in\{c,d\}$ and $j\in\{1,2\}$.  

\paragraph{The case $\spex_c=\{z^0\}$ where $z^0=(2.2,-2)^\intercal$}\hphantom{As $\pi_1(z^0)>0$ we can directly compute a first stopping integer for $\omega_{c}^1$.}
As $\pi_1(z^0)>0$ we can compute directly a first stopping integer for $\omega_{c}^1$ and:
\[
\mathfrak{F}_{\omega_{c}^1}(0,h_{\spev,r},N_\speh^{8.84}(\spev))=\dfrac{2\ln(2.2)-2\ln(\norm{\speh(z^0)})-1}{\ln\left(\speg(\spef(8.84))\right)}\simeq 20.187
\]
This means that we must compare the first 20 terms of $\omega_{c}^1$. For $\omega_{c}^2$, we cannot compute a stopping integer because $\pi_2(z^0)\leq 0$. 

Then, we compute $\speh(z^0)$ (actually already computed since we need this value to compute $\mathfrak{F}_{\omega_{c}^1}$). Since
\[
\speh(z^0)=\begin{pmatrix}2.406\\-1.685\end{pmatrix}
\]
we can update the stopping integer for $\omega_{c}^1$ and we have:
\[
\mathfrak{F}_{\omega_{c}^1}(1,h_{c},N_\speh^{8.84}(\spev))=\dfrac{2\ln(2.406)-2\ln(\norm{\speh(z^0)}-1}{\ln\left(\speg(\spef(8.84))\right)}\simeq 17.897
\]
This means that we finally compare the first 17 terms of $\omega_{c}^1$. The second coordinate is still negative; thus, we cannot compute a stopping integer. Now, as 
\[
\speh^2(z^0)\approx\begin{pmatrix}  
2.50476\\
  -1.30591\end{pmatrix}
\]
we have 
\[
\mathfrak{F}_{\omega_{c}^1}(2,h_{c},N_\speh^{8.84}(\spev))=\dfrac{2\ln(\pi_1(\speh^2(z^0)))-2\ln(\norm{\speh(z^0)}-1}{\ln\left(\speg(\spef(8.84))\right)}\simeq 16.867
\]
Then comparing the 16 first terms of $\omega_c^1$ leads to the conclusion that $\omega_{c,\rm opt}^1=\pi_1(\speh^2(z^0))$ and $\agmx(\omega_c^1)=\{2\}$. 

For $\omega_{c}^2$, we have to wait until $k=5$ to obtain $\pi_2(\speh^k(z^0))>0$. Then, we have:
\[
\speh^5(z^0)\approx \begin{pmatrix}
 0.37921\\
   0.15155
\end{pmatrix}
\]
and
\[
\mathfrak{F}_{\omega_{c}^2}(5,h_{c},N_\speh^{8.84}(\spev))=\dfrac{2\ln(\pi_2(\speh^5(z^0)))-2\ln(\norm{\speh(z^0)}-1}{\ln\left(\speg(\spef(8.84))\right)}\simeq 88.6294
\]
We must compare the 88 first terms of $\omega_c^2$ to be guaranteed to pass the maximizer rank. However, the terms $\omega_{c,k}^2$ after $k=5$ are strictly smaller than $\pi_2(\speh^5(z^0))$ (even the positive ones) and we conclude that $\omega_{c,\rm opt}^2=\pi_2(\speh^5(z^0))$ and $\agmx(\omega_c^2)=\{5\}$.

\paragraph{The case $\spex_d=\{z^1,z^2\}$ where $z^1=(-2.3,-1.9)^\intercal$ and $z^2=(-2.5,-1.5)^\intercal$}\hphantom{ and we write $r^d=\max\{\norm{z^1}^2,\norm{z^2}\}=8.9>\urho$.}

As both vectors have negative coordinates, we must wait for the first integers $k_1$, $k_2$ such that $\omega_{d,k_1}^1>0$ and $\omega_{d,k_2}^2>0$. We have, $\min\{k\in\nn : \pi_1(T^k(z^1))>0\}=6$, $\min\{k\in\nn : \pi_2(T^k(z^1))>0\}=7$, $\min\{k\in\nn : \pi_1(T^k(z^2))>0\}=4$ and $\min\{k\in\nn : \pi_2(T^k(z^2))>0\}=5$. As $\pi_1(T^4(z^2))\approx 0.08955$ and $\pi_2(T^5(z^2))\approx 0.0309$, we can compute:
\[
\mathfrak{F}_{\omega_{d}^1}(4,h_{d},N_\speh^{8.9}(\spev))=\dfrac{2\ln(\pi_1(\speh^4(z^2)))-2\ln(\norm{\speh(z^1)}-1}{\ln\left(\speg(\spef(8.9))\right)}\simeq 268.47
\]
and
\[
\mathfrak{F}_{\omega_{d}^2}(5,h_{d},N_\speh^{8.9}(\spev))=\dfrac{2\ln(\pi_2(\speh^5(z^2)))-2\ln(\norm{\speh(z^1)}-1}{\ln\left(\speg(\spef(8.9))\right)}\simeq 339.85
\]
meaning that $\omega_{d,{\rm opt}}^1=\max\{\omega_{d,k}^1: k=0,\ldots,268\}$ and $\omega_{d,{\rm opt}}^2=\max\{\omega_{d,k}^2: k=0,\ldots,339\}$.

Then, we compute the next images of $z^1$ and $z^2$ by $\speh$. We find that $\pi_1(\speh^6(z^1))\approx 0.1512>\pi_1(\speh^4(z^2))$ and thus we update the stopping integer for $\omega_{d}^1$ and we find:
\[
\mathfrak{F}_{\omega_{d}^1}(6,h_{d},N_\speh^{8.9}(\spev))=\dfrac{2\ln(\pi_1(\speh^6(z^1)))-2\ln(\norm{\speh(z^1)}-1}{\ln\left(\speg(\spef(8.9))\right)}\simeq 233.34.
\]
The new stopping integer for $\omega_{d}^1$ is set to 233. The next values of $\omega_{d,k}^1$ are smaller than $\pi_1(\speh^6(z^1))$ and thus we conclude that $\omega_{d,\rm opt}^1=\pi_1(\speh^6(z^1))$ and $\agmx(\omega_d^1)=\{6\}$. 

The same situation holds for $\omega_{d}^2$. Indeed, $\pi_2(\speh^7(z^1))\approx 0.0435835>\pi_2(\speh^5(z^2))$ and we have
\[
\mathfrak{F}_{\omega_{d}^2}(7,h_{d},N_\speh^{8.9}(\spev))=\dfrac{2\ln(\pi_2(\speh^7(z^1)))-2\ln(\norm{\speh(z^1)}-1}{\ln\left(\speg(\spef(8.9))\right)}\simeq 316.78
\]
The new stopping integer for $\omega_{d}^2$ is set to 316. The next values of $\omega_{d,k}^2$ are smaller than $\pi_2(\speh^7(z^1))$ and thus we conclude that $\omega_{d,\rm opt}^2=\pi_2(\speh^7(z^1))$ and $\agmx(\omega_d^2)=\{7\}$.

\section{Concluding remarks}\label{sec:conclusion}
In this paper, we solved specific discrete-time peak computation problems. In general, discrete-time peak computation problems consist in searching, for a given objective function, for the reachable value of a given discrete-time system that maximizes the objective function. The optimal value of this maximization problem can be viewed as the supremum of a sequence of optimal values. The solution proposed in this paper is thus based on previous results, allowing us to compute the maximum of a real sequence. The computational method requires a particular sequence greater than the analyzed sequence. This particular sequence is the image of a positive convergent geometric sequence by a strictly increasing continuous function on $[0,1]$. In this paper, we use the structure of the problem (i.e., the dynamical system) to compute the strictly increasing continuous and the positive convergent geometric sequence. The construction of these elements requires the stability of the discrete-time system, as they are constructed from $\klcls$ certificates and Lyapunov functions. We illustrate the techniques on an example from the literature. 

Three main questions remain unsolved and are left for future work. First, how can we solve a discrete-time peak computation problem for which the underlying system diverges? Second, for nonlinear systems, how can discrete-time peak computation problems be managed with an infinite bounded initial conditions set?
Finally, how can the number of unuseful computations be reduced?  
\bibliographystyle{alpha}
\bibliography{lyapunovbibarxiv}

\begin{thebibliography}{GGLW14}

\bibitem[Adj17]{adje2017proving}
A.~Adj{\'e}.
\newblock {Proving properties on PWA systems using copositive and semidefinite
  programming}.
\newblock In {\em Numerical Software Verification: 9th International Workshop,
  NSV 2016, Toronto, ON, Canada, July 17-18, 2016, Revised Selected Papers 9},
  pages 15--30. Springer, 2017.

\bibitem[Adj21]{DBLP:journals/jota/Adje21}
A.~Adj{\'{e}}.
\newblock {Quadratic maximization of reachable values of affine systems with
  diagonalizable matrix}.
\newblock {\em J. Optim. Theory Appl.}, 189(1):136--163, 2021.

\bibitem[Adj24]{adje2024parametric}
A.~Adj{\'e}.
\newblock {A parametric optimization point-of-view of comparison functions}.
\newblock {\em arXiv preprint arXiv:2408.14440}, 2024.

\bibitem[Adj25a]{adje13052025}
A.~Adjé.
\newblock Quadratic maximization of reachable values of stable discrete-time
  affine systems.
\newblock {\em To appear in Optimization}, pages 1--52, 2025.

\bibitem[Adj25b]{adje2025maximizationrealsequences}
Assalé Adjé.
\newblock On the maximization of real sequences, 2025.

\bibitem[AG24]{ahmadi2024robust}
A.~A. Ahmadi and O.~G{\"u}nl{\"u}k.
\newblock Robust-to-dynamics optimization.
\newblock {\em Mathematics of Operations Research}, 2024.

\bibitem[AGM15]{adje2015property}
A.~Adj{\'e}, P.-L. Garoche, and V.~Magron.
\newblock {Property-based polynomial invariant generation using sums-of-squares
  optimization}.
\newblock In {\em Static Analysis: 22nd International Symposium, SAS 2015,
  Saint-Malo, France, September 9-11, 2015, Proceedings 22}, pages 235--251.
  Springer, 2015.

\bibitem[AL11]{anjos2011handbook}
M.~F. Anjos and J.B Lasserre.
\newblock {\em {Handbook on semidefinite, conic and polynomial optimization}},
  volume 166.
\newblock Springer Science \& Business Media, 2011.

\bibitem[APS18]{ahiyevich2018upper}
U.~M. Ahiyevich, S.~E. Parsegov, and P.~S. Shcherbakov.
\newblock {Upper bounds on peaks in discrete-time linear systems}.
\newblock {\em Automation and Remote Control}, 79:1976--1988, 2018.

\bibitem[ATP10]{doi:10.1080/10236190902817844}
A.~Loria A.R.~Teel, D.~Nešić and E.~Panteley.
\newblock {Summability characterizations of uniform exponential and asymptotic
  stability of sets for difference inclusions}.
\newblock {\em Journal of Difference Equations and Applications},
  16(2-3):173--194, 2010.

\bibitem[Ber15]{bertsekas2015convex}
D.~Bertsekas.
\newblock {\em {Convex optimization algorithms}}.
\newblock Athena Scientific, 2015.

\bibitem[BTN01]{ben2001lectures}
A.~Ben-Tal and A.~Nemirovski.
\newblock {\em {Lectures on modern convex optimization: analysis, algorithms,
  and engineering applications}}.
\newblock SIAM, 2001.

\bibitem[BYG17]{belta2017formal}
C.~Belta, B.~Yordanov, and E.~A. Gol.
\newblock {\em {Formal methods for fiscrete-time dynamical systems}},
  volume~89.
\newblock Springer, 2017.

\bibitem[Ela05]{elaydiintroduction}
S.~Elaydi.
\newblock {An introduction to difference equations}.
\newblock 2005.

\bibitem[Els12]{ELSAYED2012378}
E.M. Elsayed.
\newblock Solutions of rational difference systems of order two.
\newblock {\em Mathematical and Computer Modelling}, 55(3):378--384, 2012.

\bibitem[GGLW14]{GEISELHART201449}
R.~Geiselhart, R.~H. Gielen, M.~Lazar, and F.~R. Wirth.
\newblock {An alternative converse Lyapunov theorem for discrete-time systems}.
\newblock {\em Systems \& Control Letters}, 70:49--59, 2014.

\bibitem[GH15]{giesl2015review}
P.~Giesl and S.~Hafstein.
\newblock {Review on computational methods for Lyapunov functions}.
\newblock {\em Discrete and Continuous Dynamical Systems-B}, 20(8):2291--2331,
  2015.

\bibitem[GRR10]{gutierrez2010}
M~Ranferi Gutiérrez, M~A Reyes, and H~C Rosu.
\newblock A note on verhulst's logistic equation and related logistic maps.
\newblock {\em Journal of Physics A: Mathematical and Theoretical},
  43(20):205204, apr 2010.

\bibitem[HR18]{articleH}
Yacine Halim and Julius~Fergy Rabago.
\newblock On the solutions of a second-order difference equation in terms of
  generalized padovan sequences.
\newblock {\em Mathematica Slovaca}, 68, 06 2018.

\bibitem[Kel14]{DBLP:journals/mcss/Kellett14}
C.~M. Kellett.
\newblock {A Compendium of comparison function results}.
\newblock {\em Math. Control. Signals Syst.}, 26(3):339--374, 2014.

\bibitem[Kel15]{kelletdcds15}
Christopher~M. Kellett.
\newblock Classical converse theorems in lyapunov's second method.
\newblock {\em Discrete and Continuous Dynamical Systems - B},
  20(8):2333--2360, 2015.

\bibitem[Kha02]{khalil2002nonlinear}
H.K. Khalil.
\newblock {\em {Nonlinear systems}}.
\newblock Pearson Education. Prentice Hall, 3rd edition, 2002.

\bibitem[KP01]{kelley2001difference}
W.~G. Kelley and A.~C. Peterson.
\newblock {\em {Difference equations: an introduction with applications}}.
\newblock Academic press, 2001.

\bibitem[Las15]{Lasserre_2015}
J.-B. Lasserre.
\newblock {\em {An introduction to polynomial and semi-algebraic
  optimization}}.
\newblock Cambridge Texts in Applied Mathematics. Cambridge University Press,
  2015.

\bibitem[Laz06]{lazarthesis}
M.~Lazar.
\newblock {\em Model predictive control of hybrid systems: stability and
  robustness}.
\newblock Phd thesis, T.U. Eindhoven, Sept 2006.

\bibitem[LHK14a]{li2014computation}
H.~Li, S.~Hafstein, and C.~M. Kellett.
\newblock Computation of lyapunov functions for discrete-time systems using the
  yoshizawa construction.
\newblock In {\em 53rd IEEE conference on decision and control}, pages
  5512--5517. IEEE, 2014.

\bibitem[LHK14b]{7040251}
H.~Li, S.~Hafstein, and C.~M. Kellett.
\newblock {Computation of Lyapunov functions for discrete-time systems using
  the Yoshizawa construction}.
\newblock In {\em 53rd IEEE Conference on Decision and Control}, pages
  5512--5517, 2014.

\bibitem[MHL15]{magron2015semidefinite}
Victor Magron, Didier Henrion, and Jean-Bernard Lasserre.
\newblock Semidefinite approximations of projections and polynomial images of
  semialgebraic sets.
\newblock {\em SIAM Journal on Optimization}, 25(4):2143--2164, 2015.

\bibitem[MHS20]{miller2020peaksafety}
J.~Miller, D.~Henrion, and M.~Sznaier.
\newblock {Peak estimation recovery and safety analysis}.
\newblock {\em IEEE Control Systems Letters}, 5(6):1982--1987, 2020.

\bibitem[NT04]{NESIC20041025}
D.~Nešić and A.~R. Teel.
\newblock {Matrosov theorem for parameterized families of discrete-time
  systems}.
\newblock {\em Automatica}, 40(6):1025--1034, 2004.

\bibitem[RKML06]{rakovic2006reachability}
S.~V. Rakovic, E.~C. Kerrigan, D.~Q. Mayne, and J.~Lygeros.
\newblock {Reachability analysis of discrete-time systems with disturbances}.
\newblock {\em IEEE Transactions on Automatic Control}, 51(4):546--561, 2006.

\bibitem[Son98]{SONTAG199893}
E.~D. Sontag.
\newblock Comments on integral variants of {ISS}.
\newblock {\em Systems \& Control Letters}, 34(1):93--100, 1998.

\bibitem[Ste14]{stevic2014representation}
Stevo Stevic.
\newblock Representation of solutions of bilinear difference equations in terms
  of generalized fibonacci sequences.
\newblock {\em Electron. J. Qual. Theory Differ. Equ}, 67(1):15, 2014.

\end{thebibliography}

\newpage

\section*{Appendix}

\begin{lemma}
Let us consider the map $\speh$ in~\eqref{runndyn}. Let $x\in\rr^2$ be a nonzero vector such that $\norm{\speh^k(x)}$ tends to 0 as $k$ tends to $+\infty$. There exists $k_1,k_2\in\nn$ such that $\pi_1(\speh^{k_1}(x))>0$ and $\pi_2(\speh^{k_2}(x))>0$.
\end{lemma}

\begin{proof}
Actually, it suffices to prove that for all $x$ in the open unit ball, there exists $k_1,k_2\in\nn$, such that $\pi_1(\speh^{k_1}(x))>0$ and $\pi_2(\speh^{k_2}(x))>0$. Indeed, if $\norm{\speh^k(x)}$ converges to 0, then $\speh^k(x)$ eventually enters the open unit ball. The worst situation is when $\speh^k(x)$ has only negative coordinates before entering the open unit ball.

Therefore, let us consider $z\in\rr^2$ not being 0 and such that $\norm{z}<1$. First, we consider $\pi_1$. If $z_1>0$ the result obviously holds. Therefore, we suppose that $z_1\leq 0$. If $z_2\leq 0$ (with $z_2\neq 0$ if $z_1=0$), as $\norm{z}<1$, we have $\pi_1(\speh(z))>0$. Now, suppose that $z_2>0$. Since $z_1<0$, we have $\pi_2(\speh(z))<0$. Recall that for all $y\in \cball{\sqrt{r}}$ where $r<\orho$, $\norm{\speh(y)}^2<\norm{y}^2$. 
If $\pi_1(\speh(z))>0$, the result follows. Thus, if $\pi_1(\speh(z))\leq 0$, we get $\pi_1(\speh^2(z))>0$ since $\norm{\speh(z)}<1$ and $\pi_1(\speh(z))<0$ and $\pi_2(\speh(z))<0$. 
 
Now, we prove the result for $\pi_2$. The result trivially holds if $z_2>0$. So, we suppose that $z_2\leq 0$. If $z_1\geq 0$ ($z_1\neq 0$ if $z_2=0$), then, as $\norm{z}<1$, we get $\pi_2(\speh(z))>0$. Now, assume that $z_1<0$. Then, it leads to $\pi_1(\speh(z))>0$. If $\pi_2(\speh(z))>0$, the results holds. Assuming that $\pi_2(\speh(z))<0$, then from the fact that $\norm{\speh(z)}<\norm{z}<1$, $\pi_1(\speh(z))>0$ and $\pi_2(\speh(z))<0$, we get $\pi_2(\speh^2(z))>0$.

\end{proof}



\end{document}